\documentclass{amsart}
\usepackage{amsmath}
\usepackage{paralist}
\usepackage{graphics}
\usepackage{epsfig}
\usepackage{graphicx}
\usepackage{epstopdf}
\usepackage[colorlinks=true]{hyperref}
\hypersetup{urlcolor=blue, citecolor=red}

\newtheorem{lemma}{Lemma}
\newtheorem{proposition}{Proposition}
\newtheorem{theorem}{Theorem}

\newtheorem{problem}{Problem}
\newtheorem*{conjecture}{Conjecture}

\theoremstyle{definition}
\newtheorem{definition}{Definition}
\newtheorem{remark}{Remark}

\title[$\varepsilon$-neighborhoods of orbits of parabolic diffeomorphisms]
{$\varepsilon$-neighborhoods of orbits and\\ formal classification of \\ parabolic diffeomorphisms\footnote{This paper was supported by the Franco-Croatian PHC-COGITO project 24710UJ M.}}

\author[Maja Resman]{Maja Resman}

\email{maja.resman@fer.hr}

\begin{document}

\maketitle

\begin{abstract}
In this article we study the dynamics generated by germs of parabolic diffeomorphisms $f:(\mathbb{C},0)\rightarrow (\mathbb{C},0)$ tangent to the identity. We show how formal classification of a given parabolic diffeomorphism can be deduced from the asymptotic development of what we call directed area of the $\varepsilon$-neighborhood of any orbit near the origin. Relevant coefficients and constants in the development have a geometric meaning. They present fractal properties of the orbit, namely its box dimension, Minkowski content and what we call residual content. \end{abstract}

\section{Introduction}\label{sec1}
\subsection{Motivation}
Each germ of a parabolic diffeomorphism in the complex plane,
\begin{equation}\label{diffe}
f(z)=z+a_1z^{k+1}+a_2z^{k+2}+\ldots,\ k\in\mathbb{N},\ a_i\in\mathbb{C},\ a_1\neq 0,
\end{equation} can, by formal changes of variables, be reduced to the standard formal normal form
\begin{equation}\label{fnf1}
f_0(z)=z+z^{k+1}+az^{2k+1},
\end{equation}
for an appropriate choice of $k\in\mathbb{N}$ and $a\in\mathbb{C}$. The formal type of a parabolic diffeomorphism is given by the pair $(k,a),\ k\in\mathbb{N},\ a\in\mathbb{C}$.

The article is motivated by the following problem:
\begin{problem}
Can we recognize a diffeomorphism by looking at one of its orbits?
\end{problem}

More precisely, the idea is to read the formal type of a diffeomorphism from the fractal properties of any orbit.

By fractal properties, we mean box dimension and Minkowski content of the orbit, which by definition are computed from the rate of growth of the area  of $\varepsilon$-neighborhoods of the orbit, as $\varepsilon\to 0$. More generally, we can refer to the asymptotic development of this area as a fractal property of the orbit. Using only the area, we noticed that the answer to the above question is negative. From the asymptotic development of the area of an $\varepsilon$-neighborhood, only information on real part of $a\in\mathbb{C}$ can be obtained. Therefore we first generalize the notion of area of the $\varepsilon$-neighborhood of the orbit to be a complex number whose modulus is the area and whose argument refers to the direction of the orbit in the plane. We call it the \emph{directed area}. We show in the article that formal type can be read from two coefficients and the leading exponent in the asymptotic development of the directed area of the $\varepsilon$-neighborhood of any orbit, as $\varepsilon\to 0$.

In some sense this problem is similar to the famous problem about hearing the shape of a drum, see Section~\ref{sec4}.

Let us comment on applications. We show that the directed area of the $\varepsilon$-neighborhood of any orbit has an asymptotic development in the scale\newline $\{\varepsilon^{1+\frac{1}{k+1}},\varepsilon^{1+\frac{2}{k+1}},\ldots,\varepsilon^{1+\frac{k}{k+1}},\varepsilon^2\log\varepsilon,\ldots \}$. One can study numerically the directed areas of $\varepsilon$-neighborhoods of just one orbit, for small $\varepsilon$. By comparing them to the scale above, one obtains relevant coefficients and concludes the formal normal form of the diffeomorphism.

It is natural to ask now the converse question.
\begin{problem}
If we only know the formal type of a given diffeomorphism, can we uniquely determine box dimension and Minkowski content of its orbits?
\end{problem}
Equivalently, we can ask if \emph{all the diffeomorphisms inside one formal class have the same fractal properties}. It turns out that the box dimension is invariant by the changes of variables inside the formal class. This is not the case for Minkowski content and what we call residual content if we work with general formal changes of variables. However, the problem is solved if we restrict the definition of formal equivalence relation and allow only formal changes of variables that are tangent to the identity. In this sense, each parabolic germ \eqref{diffe} is formally conjugate to a simpler germ
\begin{equation}\label{expfnf}
g_0(z)=z+a_1z^{k+1}+a_1^2 \cdot a \cdot z^{2k+1},
\end{equation}
where $a_1\in\mathbb{C}$ is the first coefficient of the initial diffeomorphism~\eqref{diffe} and $a,\ k$ are as in \eqref{fnf1}. We call \eqref{expfnf} the \emph{extended formal normal form}.
The formal type of a diffeomorphism is thus not completely described by the pair $(k,a)$, but by the triple $(k,a_1,a)$.

Finally, with this restricted notion of formal equivalence, we get a bijective correspondence between the formal type of a diffeomorphism and
the fractal properties of its orbits. We comment on the prospects and problems of analytic classification of diffeomorphisms using fractal properties of their orbits in Section~\ref{sec4}.

\smallskip
In the real case, a similar idea that fractal properties of an orbit near a fixed point, i.e. the rate of growth of its $\varepsilon$-neighborhoods, carry some information on the properties of the generating function itself, was discussed before in e.g. \cite{MaReZu} and \cite{overview}. A bijective correspondence was found between the multiplicity of a fixed point of a function on the real line and the rate of growth of $\varepsilon$-neighborhoods of any orbit.

\medskip

\subsection{Definitions and notations}
Let us recall precisely the main definitions and notations we use in this article.

Let $f:\mathbb{C}\to \mathbb{C}$, $f\in \text{\it Diff}(\mathbb{C},0)$, be a germ of a diffeomorphism fixing the origin. We say that the germ $f$ is \emph{parabolic} if $f'(0)=1$. If $f'(0)=\exp{(2\pi i\lambda)}$, where $\lambda\in\mathbb{Q}$, the diffeomorphism can be reduced to the previous case by considering its higher iterates, but we will not discuss it in this article. Therefore, in the neighborhood of the origin, we suppose that $f(z)$ is of the form
$$
f(z)=z+a_1 z^{k+1}+a_2z^{k+2}+a_3z^{k+3}+o(z^{k+3}),
$$
where $a_1\in \mathbb{C}^*$, $a_i\in\mathbb{C}$, $i=2,3,\ldots$ and $k\in \mathbb{N},\ k\geq 1$.

By $S^{f}(z_0)$, we denote the orbit generated by $f(z)$ with the initial point $z_0$ in the neighborhood of the origin, $S^f(z_0)=\{z_n=f^{\circ n}(z_0)|n\in\mathbb{N}_0\}$. Near the origin, such orbits form the so-called \emph{Leau-Fatou flower}, see e.g. \cite{milnor} or \cite{loray}. In short, there exist $k$ attracting and $k$ repelling sectors, called petals, around equidistant repelling and attracting directions. Attracting and repelling directions are normalized complex numbers $(-a_1)^{-1/k}$,\ $a_1^{-1/k}$ respectively. Orbits are tangent to attracting or repelling directions at the origin, see Figure~\ref{petal}. In the sequel, we suppose that $z_0$ belongs to an attracting sector of the origin. Otherwise, if $z_0$ belongs to a repelling sector, we consider the inverse diffeomorphism $f^{-1}(z)$ instead.

\begin{figure}[htp]
\begin{center}
  % replace aims_logo.eps by your figure file name
  \includegraphics[width=3in]{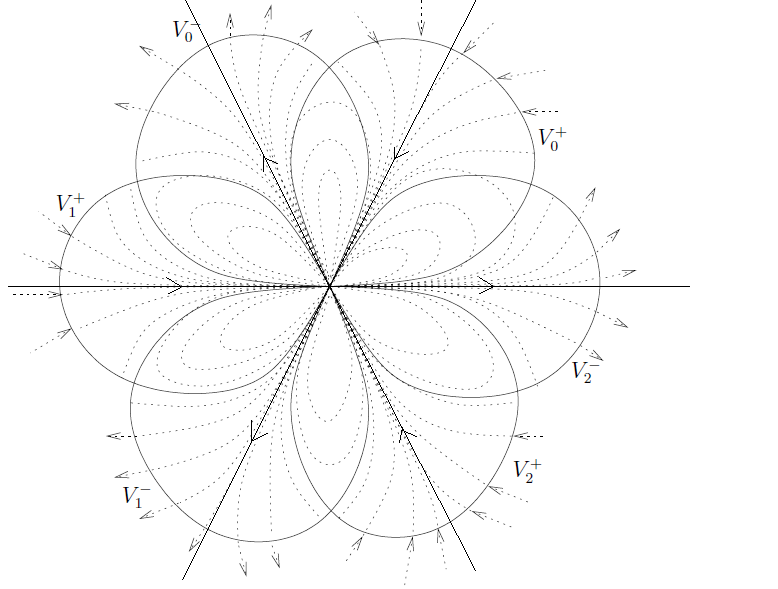}\\
  \caption{\small{Attracting and repelling sectors and directions for e.g. $f(z)=z+z^4+o(z^4)$}.}\label{petal}
  \end{center}
\end{figure}

Now we discuss some properties of measurable sets in the plane. Let $U\subset\mathbb{R}^2$, or $U\subset\mathbb{C}$, be a measurable set whose center of mass is not the origin. By $A(U)$, we denote its area. In Definition~\ref{dir} in Section \ref{sec2}, we define the \emph{directed area} $A^\mathbb{C}(U)$ of $U$ as the complex number which encodes the area of the set $U$, as well as the direction in which the set is placed. Note that the directed area does not verify the finite stability property, that is, $A^\mathbb{C}(U\cup V)=A^\mathbb{C}(U)+A^\mathbb{C}(V)$, for disjoint sets $U$ and $V$. Thus it does not satisfy the properties of a vector measure defined in e.g. \cite{klu}. Moreover, this notion should not be confused with directional $\varepsilon$-neighborhood, also called directional Minkowski sausage, and directional Minkowski content, defined in \cite{tricot}.

The fractal properties of a set $U$ are related to the asymptotic behavior of the area of its $\varepsilon$-neighborhood, denoted $U_\varepsilon$, as $\varepsilon\to 0$. The growth rate of the area $A(U_\varepsilon)$ of the $\varepsilon$-neighborhoods reveals the \emph{density} of the set. It is measured by box dimension and Minkowski content of $U$. We recall the definitions of the Minkowski content and the box dimension of a measurable set $U\subset\mathbb{R}^2$ (or $\mathbb{C}$).

By {\it lower and upper $s$-dimensional  Minkowski content of $U$}, $0\leq s\leq 2$, we mean
$$
{\mathcal M}_{*}^{s} (U)=\liminf_{{\varepsilon}\to0}\frac{A(U_\varepsilon)}{{\varepsilon}^{2-s}}\text{\ \  and\ \ }{\mathcal M}^{*s} (U)=\limsup_{{\varepsilon}\to 0}\frac{A(U_\varepsilon)}{{\varepsilon}^{2-s}}
$$
respectively. Furthermore, \emph{lower and upper box dimension of $U$} are defined by
$$
\underline{\dim}_\text{\it B} U=\inf\{s\geq0\  |\ {\mathcal M}_{*}^s(U)=0\},\ \overline{\dim}_\text{\it B} U=\inf\{s\geq0\  |\ {\mathcal M}^{*s}(U)=0\}.
$$
As functions of $s\in[0,2]$, $\mathcal{M}^{*s}(U)$ and $\mathcal{M}_*^s(U)$ are step functions that jump only once from $+\infty$ to zero as $s$ grows, and upper or lower box dimension are equal to the value of $s$ when jump in upper or lower content appears.

If $\underline{\dim}_B U=\overline{\dim}_\text{\it B} U$, then we put $\dim_\text{\it B}(U)=\underline{\dim}_B U=\overline{\dim}_\text{\it B} U$ and call it the \emph{box dimension of $U$}. In literature, the upper box dimension of $U$ is also referred to as the \emph{limit capacity} of $U$, see \cite{PT}.

If $d=\dim_B(U)$ and, moreover, $\mathcal{M}_{*}^d(U)=\mathcal{M}^{*d}(U)\in(0,\infty)$, we say that the set $U$ is \emph{Minkowski measurable}. In that case, we denote the common value of the Minkowski contents simply by $\mathcal{M}(U)$, and call it the \emph{Minkowski content of $U$}.

In short, if $A(U_\varepsilon)\sim C \varepsilon^{2-s}$, as $\varepsilon\to 0$, for some $s\in[0,2]$, $C>0$, in the sense that $\lim_{\varepsilon\to 0}\frac{A(U_\varepsilon)}{\varepsilon^{2-s}}=C$, then $\dim_{\text{\it B}}(U)=s$ and $U$ is Minkowski measurable with Minkowski content $\mathcal{M}(U)=C$.
For more details on box dimension, see Falconer \cite{falconer} or Tricot \cite{tricot}.

In this article, previous definitions are considered for an orbit of a parabolic diffeomorphism which accumulates at the origin, $U=S^f(z_0)$. The asymptotic behavior of the directed area of its $\varepsilon$-neighborhood, $A^\mathbb{C}(S^f(z_0)_\varepsilon)$, carries information on the density of accumulation, as well as on the direction of approach to the origin.

At the end of this Section, we state Proposition 1.3.1. from \cite{loray} about the formal classification of germs of parabolic diffeomorphisms. We say that two germs of diffeomorphisms, $f,\ g\in \text{\it Diff}(\mathbb{C},0)$, are \emph{formally conjugate} if there exists a formal change of variables $\phi(z)=\sum_{l=1}^{\infty}c_l z^l,\ c_l\in\mathbb{C}$, possibly divergent, which transforms $f$ to $g$, i.e. $g(z)=\phi^{-1}\circ f\circ \phi (z)$.

\begin{proposition}[Formal classification of parabolic diffeomorphisms, Proposition 1.3.1 in \cite{loray}]\label{ch1}
Let \begin{equation*}
f(z)=z+a_1z^{k+1}+a_2 z^{k+2}+\ldots,\quad a_1\in\mathbb{C^*},\ a_i\in \mathbb{C},\ i=2,3,\ldots,\quad k\in\mathbb{N},\ k>1,
\end{equation*} be a germ of a parabolic diffeomorphism. By a formal change of variables, it can be transformed to its formal normal form
\begin{equation}\label{sfnf}
f_0(z)=z+z^{k+1}+az^{2k+1},\ a\in\mathbb{C}.
\end{equation}
\end{proposition}
%One can alternatively consider the time-one map of the vector field $X_{k,a}=\frac{z^{k+1}}{1+(\frac{k+1}{2}-a)z^k}\frac{\partial}{\partial z}$ as the formal normal form, since $f_0(z)$ is formally equivalent to the time-one map $\text{exp}(X_{k,a})(z)$. For more details, see \cite{loray}.

The formal type of a diffeomorphism $f$ is thus given by the pair $(k,\ a)$. In the formal change of variables, the multiplicity $k$ remains unchanged. The final coefficient $a$ is equal to $a=Res(\frac{1}{f(z)-z},0)$, where $Res(\frac{1}{f(z)-z},0)$ remains unchanged in the formal change of variables, for proof see \cite{milnor}. In literature, this invariant residue is called \emph{residual index of fixed point zero} and denoted by $i(f,0)$.

The following remark is important for the sequel. We will see in Section~\ref{sec2} that formal changes of variables which are not tangent to the identity, i.e. which are of the type $\phi(z)=c_1 z+c_2 z^2+o(z^3)$, where $c_1\neq 1$, affect the fractal properties, namely the Minkowski content, of the diffeomorphism. Therefore, if we want the whole formal class of the diffeomorphism $f$, including its formal normal form $f_0$, to have the same fractal properties, we allow only formal changes of variables tangent to the identity, i.e. of the type
$$
\phi(z)=z+c_2z^2+c_3z^3+o(z^3).
$$
It is easy to check that, if we make formal changes of variables tangent to the identity, the first coefficient $a_1$ of the diffeomorphism obviously remains unchanged, so we have the following version of Proposition~\ref{ch1}:
\begin{proposition}\label{ch2}
Let $f(z),\ a,\ k$ be as in Proposition 1.3.1. By formal changes of variables tangent to the identity, $f(z)$ can be transformed to its formal normal form
\begin{equation}\label{exfnf}
g_0(z)=z+a_1z^{k+1}+a_1^2\cdot a\cdot z^{2k+1}.
\end{equation}
\end{proposition}
In this case, the formal type of a diffeomorphism $f$ is given by the triple $(k,\ a,\ a_1)$.

Note that $f_0$ is obtained from $g_0$ if we apply one more change of variables $\phi(z)=\lambda z,\ \lambda^k=a_1$, which is not tangent to the identity and which eliminates the coefficient $a_1$. To avoid confusion, in the rest of the article, the formal normal form \eqref{sfnf} will be called the \emph{standard formal normal form} and \eqref{exfnf} the \emph{extended formal normal form}, since it carries more information on the initial diffeomorphism.

Let us note that in both cases only first $k+1$ coefficients of a diffeomorphism contribute to its formal normal forms. All the monomials of order greater than $2k+1$, possibly infinitely many of them, can be eliminated one by one by a formal change of variables, without affecting the former coefficients. Therefore, they are of no importance for formal classification, and in the sequel we can restrict ourselves to the diffeomorphism up to the order $2k+1$, $$f(z)=z+a_1z^{k+1}+a_2z^{k+2}+\ldots+a_{k+1}z^{2k+1}+o(z^{2k+1}).$$
Accordingly, to relate formal normal form of a diffeomorphism with coefficients in the formal asymptotic development of the directed area of the $\varepsilon$-neighborhood of an orbit, it will suffice to study only the first $k+1$ coefficients of the formal asymptotic development.

\section{Main results}\label{sec2}

Let us define the directed area and the directed Minkowski content of a measurable set.

\begin{definition}[directed area of a measurable set]\label{dir}
Let $U\subset\mathbb{C}$ be a measurable set, whose center of mass is not the origin. We define the \emph{directed area of the set $U$}, denoted by $A^\mathbb{C}(U)$, as the complex number
$$
A^{\mathbb{C}}(U)=A(U)\cdot \nu_{\mathbf{t}(U)},
$$
where $A(U)$ denotes the area of $U$, $\mathbf{t}(U)\in\mathbb{C}$ the center of mass of $U$ and $\nu_{\mathbf{t}(U)}=\frac{\mathbf{t}(U)}{|\mathbf{t}(U)|}\in\mathbb{C}$, $|\nu_{\mathbf{t}(U)}|=1$, the normalized center of mass of $U$.
\end{definition}

Minkowski content of a measurable set $U\subset\mathbb{C}$ is by definition in Section~\ref{sec1} equal to the the first coefficient in the asymptotic development of the area of the $\varepsilon$-neighborhood of U. We define \emph{directed Minkowski content} analogously, but using the directed area of the $\varepsilon$-neighborhood, thus taking into account the position of the set in the plane.

Let $U_\varepsilon$ denote the $\varepsilon$-neighborhood of $U$.
\begin{definition}[Directed Minkowski content of a measurable set] \label{ccontent}
Let $U\subset\mathbb{C}$ be a measurable set, such that its center of mass is not the origin. Let $d=\dim_{\text{\it B}}U$. We define the \emph{directed Minkowski content of $U$}, denoted by $\mathcal{M}^{\mathbb{C}}(U)$, as the complex number
\begin{equation*}
\mathcal{M}^{\mathbb{C}}(U)=\lim_{\varepsilon\to 0}\frac{A^{\mathbb{C}}(U_\varepsilon)}{\varepsilon^{2-d}}.
\end{equation*}
\end{definition}

Note that, by definition, $|\mathcal{M}^{\mathbb{C}}(U)|=\mathcal{M}(U)$, where $\mathcal{M}(U)$ is the Minkowski content of $U$.

\medskip

Let $f:\mathbb{C}\to\mathbb{C}$ be a germ of a parabolic diffeomorphism:
$$f(z)=z+a_1z^{k+1}+o(z^{k+1}),\ \ \ a_1\in\mathbb{C}^*,\ k\in\mathbb{N}.$$
Let the initial point $z_0$ lie in an attracting sector of the origin. We denote the attracting orbit of $f$ with initial point $z_0$ by $S^f(z_0)$.

In Theorem~\ref{asy} at the end of this Section, we show that the directed area of the $\varepsilon$-neighborhood of the orbit $S^f(z_0)$ has an asymptotic development of the form:
\begin{equation}\label{asydevel}
\begin{split}
&A^{\mathbb{C}}(S^f(z_0)_\varepsilon)=K_1 \varepsilon^{1+\frac{1}{k+1}}+K_2 \varepsilon^{1+\frac{2}{k+1}}+\ldots+K_{k-1}\varepsilon^{1+\frac{k-1}{k+1}}+   \\
&\ \ \ +K_{k}\varepsilon^{1+\frac{k}{k+1}}\log\varepsilon+S\varepsilon^{1+\frac{k}{k+1}}+K_{k+1} \varepsilon^2\log\varepsilon+o(\varepsilon^2\log\varepsilon),\ \varepsilon\to 0.
\end{split}
\end{equation}
Here, coefficients $K_i$, $i=1,\ldots,k+1,$ and $S$ are complex numbers. For the precise statement and properties of coefficients $K_i$ and $S$, namely their dependence on the coefficients of the diffeomorphism and on the initial point, see Theorem~\ref{asy} at the end of this Section.
\medskip

From the development \eqref{asydevel}, it holds that $$A(S^f(z_0)_\varepsilon)=|A^{\mathbb{C}}(S^f(z_0)_\varepsilon)|=|K_1|\varepsilon^{1+\frac{1}{k+1}}+o(\varepsilon^{1+\frac{1}{k+1}}),\ \varepsilon\to 0.$$ Therefore we have that any orbit $S^f(z_0)$ is Minkowski measurable, with:
\begin{equation}\label{fra1}\dim_\text{\it B}(S^f(z_0))=1-\frac{1}{k+1},\quad \mathcal{M}(S^f(z_0))=|K_0|, \quad \mathcal{M}^{\mathbb{C}}(S^f(z_0))=K_1.\end{equation}

Motivated by the fact that the first coefficient of \eqref{asydevel} incorporates directed\linebreak Minkowski content of the orbit, we define the \emph{directed residual content} of the orbit as the coefficient in front of the logarithmic term, $\varepsilon^2\log\varepsilon$.
\begin{definition}[Directed residual content]
We define the \emph{directed residual content} $\mathcal{R}^{\mathbb{C}}(S^f(z_0))$ of the orbit $S^f(z_0)$ as the complex number
\begin{equation}\label{fra2}
\mathcal{R}^{\mathbb{C}}(S^f(z_0))=K_{k+1},
\end{equation}
where $K_{k+1}$ is the coefficient in front of the logarithmic term $\varepsilon^2\log\varepsilon$ in the development \eqref{asydevel}.
\end{definition}

\medskip
Now we state the two main results of the article. First, the standard formal normal form, $f(z)=z+z^{k+1}+az^{2k+1}$, of a given parabolic diffeomorphism can be deduced from fractal properties of one of its orbits near the origin.
\begin{theorem}[Standard formal normal form and fractal properties of an orbit]\label{fnf}\
The standard formal type $(k,\ a)$ of a parabolic diffeomorphism $f(z)$ is uniquely determined by $dim_B(S^f(z_0))$, $\mathcal{M}^\mathbb{C}(S^f(z_0))$ and $\mathcal{R}^\mathbb{C}(S^f(z_0))$ of any attracting orbit $S^f(z_0)$ near the origin.

\noindent Moreover, the following explicit formulas hold:
\begin{equation}\label{form}
\begin{split}
k\ =\ &\frac{\dim_B(S^f(z_0))}{1-\dim_B(S^f(z_0))},\\
a\ =\ &\frac{k+1}{2}-\frac{k+1}{\pi}\cdot Re\left(\frac{\mathcal{R}^\mathbb{C}(S^f(z_0))}{\nu_{\mathcal{M}^\mathbb{C}(S^f(z_0))}}\right)+\\
&\qquad +i\cdot \phi(k)\cdot\mathcal{M}(S^f(z_0))\cdot Im\left(\frac{\mathcal{R}^\mathbb{C}(S^f(z_0))}{\nu_{\mathcal{M}^\mathbb{C}(S^f(z_0))}}\right),
\end{split}
\end{equation}
where $\nu_{\mathcal{M}^\mathbb{C}(S^f(z_0))}$ is the normalized directed Minkowski content and $\phi(k)$ is a function of $k$, explicitly given by
$$
\phi(k)=\frac{k(k+1)}{k-1}\cdot \frac{1}{\sqrt\pi}\cdot\frac{\frac{\Gamma(\frac{1}{k+1})}{\Gamma(\frac{3}{2}+\frac{1}{k+1})}+\sqrt{\pi}}{\frac{\Gamma(\frac{1}{2}+\frac{1}{2k+2})}{\Gamma(2+\frac{1}{2k+2})}-\sqrt{\pi}}\cdot\frac{\Gamma(1 + \frac{1}{2k+2})}{\Gamma(\frac{3}{2}+\frac{1}{2k+2})}.
$$
Here, $\Gamma$ denotes the gamma function.
\end{theorem}
As we have discussed earlier in Section~\ref{sec1}, the converse of Theorem~\ref{fnf} is not true. Nevertheless, if we consider the extended formal normal form, $f(z)=z+a_1z^{k+1}+a_1^2\cdot a \cdot z^{2k+1}$, instead of standard formal normal form, Theorem~\ref{fnf} takes the form of the stronger equivalence statement:
\begin{theorem}[Extended formal normal form and fractal properties of an orbit]\label{fnfe}
There exists a bijective correspondence between the following triples:
\begin{itemize}
\item[(i)] the extended formal type of a diffeomorphism, $(k,\ a_1,\ a),$
\item[(ii)] $\big(\dim_\text{\it B}(S^f(z_0)),\ \mathcal{M}^\mathbb{C}(S^f(z_0)),\ \mathcal{R}^\mathbb{C}(S^f(z_0))\big),$
\end{itemize}
where $S^f(z_0)$ is any attracting orbit of a diffeomorphism.
The bijective correspondence is given by formulas \eqref{form} and the following formula for $a_1$:
\begin{equation}\label{a1}
a_1=\mathcal{M}^\mathbb{C}(S^f(z_0))^{-k}\cdot \frac{(-2)^{-k}}{\mathcal{M}(S^f(z_0))}\cdot \Big(\frac{k}{\sqrt \pi}\frac{\Gamma(\frac{3}{2}+\frac{1}{2k+2})}{\Gamma(\frac{1}{2k+2})}\Big)^{-(k+1)}.
\end{equation}
\end{theorem}

The converse states that all the attracting orbits of all the diffeomorphisms of the same extended formal type share the same fractal properties.

Actually, for the precise converse statement, we have to make the following remark about the sectorial dependence of fractal properties on the initial point of the orbit. Suppose that we know only the extended formal type of a diffeomorphism and we want to compute the directed Minkowski content and the directed residual content of any attracting orbit $S^f(z_0)$ of the diffeomorphism. The directed Minkowski content is given by reformulation of the formula \eqref{a1}:
$$
\mathcal{M}^\mathbb{C}(S^f(z_0))=\frac{k+1}{k}\cdot\sqrt\pi\cdot\frac{\Gamma(1 + \frac{1}{2k+2})}{\Gamma(\frac{3}{2}+\frac{1}{2k+2})}\bigg(\frac{2}{|a_1|}\bigg)^{1/(k+1)}\cdot \nu_A,
$$
where $\nu_A$ is the attracting direction in whose attracting sector $z_0$ lies. Therefore, the fractal properties do differ slightly in argument for the orbits in $k$ different attracting sectors, but they do not differ for the orbits inside one attracting sector. Their modules, in particular the Minkowski content, are the same in all sectors.

\begin{remark}
By \eqref{fra1} and \eqref{fra2}, we can as well express the correspondence in Theorem~\ref{fnf} and Theorem~\ref{fnfe} in terms of coefficients $K_1$, $K_{k+1}$ and the exponent $k$ in the asymptotic development of $A^\mathbb{C}(S^f(z_0)_\varepsilon)$, instead in terms of fractal properties of the orbit $S^f(z_0)$.
\end{remark}

\medskip
At the end of the Section, we state the auxiliary theorem which gives a more precise description of the coefficients in the development \eqref{asydevel}. This theorem is an important part of the proof of Theorem~\ref{fnf} and Theorem~\ref{fnfe}.

Let $f(z)=z+a_1z^{k+1}+a_2z^{k+2}+o(z^{k+2}),\ a_1\in \mathbb{C}^*$, be a parabolic diffeomorphism and let $S^f(z_0)$ be its orbit with initial point $z_0$ lying in one of $k$ attracting sectors. Let
$$
A=(-ka_1)^{-\frac{1}{k}}
$$
be the one of $k$ attracting directions in whose attracting sector the initial condition $z_0$ lies. In other words, we chose the $\frac{1}{k}-$th complex root of $-\frac{1}{a_1}$  whose argument is closest to $z_0$. By $\nu_A$, we denote the normalized complex number $A$.
\begin{theorem}[Asymptotic development of the directed area of $\varepsilon$-neighborhoods of orbits]\label{asy}
The directed area of $\varepsilon$-neighborhood of an orbit $S^f(z_0)$ has the following asymptotic development:
\begin{equation}\label{comarea}
\begin{split}
&A^{\mathbb{C}}(S^f(z_0)_\varepsilon)=K_1 \varepsilon^{1+\frac{1}{k+1}}+K_2 \varepsilon^{1+\frac{2}{k+1}}+\ldots+K_{k-1}\varepsilon^{1+\frac{k-1}{k+1}}+  \\
&\ \ \ +K_{k}\varepsilon^{1+\frac{k}{k+1}}\log\varepsilon+S(f,z_0)\varepsilon^{1+\frac{k}{k+1}}+K_{k+1} \varepsilon^2\log\varepsilon+o(\varepsilon^2\log\varepsilon),\ \varepsilon\to 0.
\end{split}
\end{equation}
All coefficients $K_i$, $i=1,\ldots,k+1,$ are complex-valued functions which depend only on $k$, $A$ and on the first $i$ coefficients $a_2,\ \ldots,\ a_i$ of the diffeomorphism. The coefficient $S(f,z_0)$ is a complex-valued function which depends on the whole diffeomorphism $f(z)$ and on the initial condition $z_0$.

Furthermore, `important' coefficients $K_1$ and $K_{k+1}$ are of the form:
\begin{equation}\label{impcoef}
\begin{split}
&K_1=\frac{k+1}{k}\cdot\sqrt\pi\cdot\frac{\Gamma(1 + \frac{1}{2k+2})}{\Gamma(\frac{3}{2}+\frac{1}{2k+2})}\bigg(\frac{2}{|a_1|}\bigg)^{1/(k+1)}\cdot \nu_A,\\
&K_{k+1}=\nu_A\cdot\Bigg[-\frac{\pi}{k+1}Re\big(\frac{a_{k+1}}{a_1^2}-\frac{k+1}{2}\big)+\Bigg.\\
&\qquad \Bigg. \bigg(\frac{2(k-1)}{k+1}\Big(\frac{|a_1|}{2}\Big)^{1/(k+1)}\frac{\frac{\Gamma(\frac{1}{2}+\frac{1}{2k+2})}{\Gamma(2+\frac{1}{2k+2})}-\sqrt{\pi}}{\frac{\Gamma(\frac{1}{k+1})}{\Gamma(\frac{3}{2}+\frac{1}{k+1})}+\sqrt{\pi}}\bigg)\cdot i\cdot Im\big(\frac{a_{k+1}}{a_1^2}-\frac{k+1}{2}\big) \Bigg]+\\
&\hspace{9cm} +g(k,A,a_2,\ldots,a_{k}).
\end{split}
\end{equation}
Here, $g(k,A,a_2,\ldots,a_k)$ is a complex-valued function with the property\linebreak $g(k,A,0,\ldots,0)=0$.
\end{theorem}
Note that coefficients $K_i$, $i=1,\ldots,k+1,$ do not depend on the initial point $z_0$, but only on the attracting sector of the initial point (via $A$). Dependence of $S$ on the initial point comes from the directed area of the tail of the $\varepsilon$-neighborhood of the orbit. On the other hand, the first $k+1$ coefficients in the development of the directed area of the nucleus are independent of the initial point, see Lemmas \ref{asynucl} to \ref{masstail} in Section~\ref{sec3}. It is interesting to observe that further coefficients in the asymptotic development depend on the initial point $z_0$.

\smallskip

\section{Proof of Theorem~\ref{asy}}\label{sec3}

We describe below the main steps of the proof of Theorem~\ref{asy}. The proof is rather long and technical, so each step is contained in a separate lemma below. Some auxiliary propositions are in the Appendix.

Suppose $\varepsilon>0$. By Definition~\ref{dir},
$$
A^{\mathbb{C}}(S^f(z_0)_\varepsilon)=A(S^f(z_0)_\varepsilon)\cdot \nu_{\mathbf{t}(S^f(z_0)_\varepsilon)}.
$$
Therefore, we need to compute the first $k+1$ terms in the development of the area of the $\varepsilon$-neighborhood and the first $k+1$ terms in the development of its normalized center of mass. Following the idea from \cite{tricot}, the $\varepsilon$-neighborhood of the orbit, $S^f(z_0)_\varepsilon$, can be regarded as a disjoint union of the nucleus $N_{\varepsilon}$ and the tail $T_{\varepsilon}$. The tail $T_{\varepsilon}$ is the union of disjoint discs $K(z_i,\varepsilon),\ i=0,\ldots,n_\varepsilon$. The nucleus $N_{\varepsilon}$ is the union of overlapping discs $K(z_i,\varepsilon)$, $i=n_\varepsilon+1,\ldots,\infty$. Here, $n_\varepsilon$ denotes the index when discs around the points start to overlap. In our case, this `critical' index $n_\varepsilon$ is unique and well-defined, since the distances between two consecutive points are strictly decreasing, see Proposition~\ref{nhood}.$i)$ in the Appendix.

\emph{Step 1.} In Lemma~\ref{asyneps}, we compute the first $k+1$ terms in the asymptotic development of the index $n_\varepsilon$, as $\varepsilon\to 0$.

\emph{Step 2. } Using the development for $n_\varepsilon$, we compute the first $k+1$ terms in the development of the area of the $\varepsilon$-neighborhood of the orbit, $A(S^f(z_0)_\varepsilon)$, as $\varepsilon\to 0$. This consists of two parts: first, in Lemma~\ref{asynucl}, we compute the development of the area of the nucleus, $A(N_\varepsilon)$. Second, in Lemma~\ref{asytail}, we compute the development of the area of the tail, $A(T_\varepsilon)$. Finally,
\begin{equation}\label{sumarea}
A(S^f(z_0)_\varepsilon)=A(N_\varepsilon)+A(T_\varepsilon).
\end{equation}

\emph{Step 3. } We need to find first $k+1$ terms in the development of the normalized center of mass of the $\varepsilon$-neighborhood of the orbit, $\nu_{\mathbf{t}(S^f(z_0)_\varepsilon)}$, as $\varepsilon\to 0$.
Obviously,
\begin{equation}\label{summass}
\nu_{\mathbf{t}(S^f(z_0)_\varepsilon)}=\frac{\mathbf{t}(N_\varepsilon)\cdot A(N_\varepsilon)+\mathbf{t}(T_\varepsilon)\cdot A(T_\varepsilon)}{|\mathbf{t}(N_\varepsilon)\cdot A(N_\varepsilon)+\mathbf{t}(T_\varepsilon)\cdot A(T_\varepsilon)|}.
\end{equation}
Again, in Lemma~\ref{massnucl}, we compute first $k+1$ terms for the nucleus, $\mathbf{t}(N_\varepsilon)\cdot A(N_\varepsilon)$. In Lemma~\ref{masstail}, we do the same for the tail, $\mathbf{t}(T_\varepsilon)\cdot A(T_\varepsilon)$.

Now, combining the obtained developments \eqref{sumarea} and \eqref{summass}, the development for $A^\mathbb{C}(S^f(z_0)_\varepsilon)$ follows.

\medskip
The following lemmas are used in the proof. They provide asymptotic developments up to the first $k+1$ terms of the expressions that are neccessary for computing the first $k+1$ terms of asymptotic development of the directed area. In all these developments, we provide precise information only on the first and on the $(k+1)$-st coefficient, since they are the only ones that affect the first and $(k+1)$-st coefficient in the development of the final directed area.
\begin{lemma}[Asymptotic development of $n_\varepsilon$]\label{asyneps}
Suppose $n_\varepsilon$ is the critical index separating the nucleus and the tail, as in the proof above. Then it has the following asymptotic development:
\begin{equation}\label{devneps}
n_\varepsilon=p_1\varepsilon^{-1+\frac{1}{k+1}}+p_2\varepsilon^{-1+\frac{2}{k+1}}+
p_3\varepsilon^{-1+\frac{3}{k+1}}+\ldots+p_k\varepsilon^{-1+\frac{k}{k+1}}+\log\varepsilon+o(\log\varepsilon),\ \varepsilon\to 0,
\end{equation}
where coefficients $p_i=p_i(k,A,a_2,\ldots,a_i)$, $i=1,\ldots,k+1,$ are real-valued functions of $k$ and first $i$ coefficients of $f(z)$. Moreover,
\begin{align*}
p_1=&\Big(\frac{2}{|a_1A^{k+1}|}\Big)^{-1+\frac{1}{k+1}},\\ p_{k+1}=&\frac{k}{k+1}Re\Big[\Big(\frac{a_{k+1}}{a_1}-\frac{(k+1)a_1}{2}\Big)A^{k}\Big]+g(k,A;a_2,\ldots,a_k),
\end{align*}
where $g=g(k,A,a_2,\ldots,a_k)$ is a real-valued function which satisfies $g(k,A,0,\ldots,0)\break =0$.
\end{lemma}

\begin{proof}
By $d_n=|z_{n+1}-z_n|$,\ $n\in\mathbb{N}_0$, we denote the distances between two consecutive points of the orbit. The critical index $n_\varepsilon$ is then determined by the inequalities
\begin{equation}\label{neps}
d_{n_\varepsilon}<2\varepsilon,\ d_{n_\varepsilon-1}\geq 2\varepsilon.
\end{equation}

To obtain the asymptotic development of $n_\varepsilon$, we first compute asymptotic development for $d_n$, as $n\to\infty$. Using development \eqref{znasy} for $z_n$ from Proposition~\ref{orbit} in the Appendix, we get
\begin{equation}\label{znrasy}
\begin{split}
&z_{n+1}-z_n=a_1A^{k+1}n^{-1-\frac{1}{k}}+h_2n^{-1-\frac{2}{k}}+\ldots+h_k n^{-2}-\\
&\quad -\left[\Big(a_{k+1}-\frac{(k+1)a_1^2}{2}\Big)A^{2k+1}\frac{k+1}{k}+g(k,A,a_2,\ldots,a_{k})\right]n^{-2-\frac{1}{k}}\log n+\\
&\hspace{7.5cm}+o(n^{-2-\frac{1}{k}}\log n), \ n\to\infty,
\end{split}
\end{equation}
where $h_i=h_i(k,A,a_2,\ldots,a_i)$ and $g=g(k,A,a_2,\ldots,a_k)$ are complex-valued functions and $g(k,A,0,\ldots,0)=0$. Furthermore,
\begin{equation}\label{dn}
\begin{split}
&d_n=|a_1A^{k+1}|n^{-1-\frac{1}{k}}+q_2n^{-1-\frac{2}{k}}+\ldots+q_kn^{-2}-\\
&-\left[\frac{k+1}{k}|a_1A^{k+1}|Re\Big(\Big(\frac{a_{k+1}}{a_1}-\frac{(k+1)a_1}{2}\Big)A^{k}\Big)+r(k,A,a_2,..,a_{k})\right]n^{-2-\frac{1}{k}}\log n+\\
&\hspace{8cm}+o(n^{-2-\frac{1}{k}}\log n),\ n\to\infty,
\end{split}
\end{equation}
where $q_i=q_i(k,A,a_2,\ldots,a_i)$ and $r=r(k,A,a_2,\ldots,a_{k})$ are real-valued functions and $r(k,A,0,\ldots,0)=0$.

From \eqref{neps} and \eqref{dn} we deduce the asymptotic development of $n_{\varepsilon}$ as $\varepsilon\to 0$, iteratively, term by term.
\end{proof}
Note that the above proof provides developments \eqref{znrasy} and \eqref{dn} for $z_n-z_{n+1}$ and for the distances $d_n$ between two consecutive points, which we also need later.

\medskip

\begin{lemma}[Asymptotic development of the area of the nucleus]\label{asynucl}
The following asymptotic development for the area of the nucleus of the $\varepsilon$-neighborhood of the orbit holds:
\begin{equation}\label{nucleus}
\begin{split}
A(N_\varepsilon)=\frac{2^{-\frac{k}{k+1}}\sqrt\pi}{k}&\left(\frac{\Gamma(\frac{1}{2k+2})}{\Gamma(\frac{3}{2}+\frac{1}{2k+2})}-\sqrt\pi\right)|a_1|^{-\frac{1}{k+1}}\cdot \varepsilon^{1+\frac{1}{k+1}}+h_2\varepsilon^{1+\frac{2}{k+1}}+\\
&+\ldots+h_{k}^{(1)} \varepsilon^{1+\frac{k}{k+1}}+h_{k}^{(2)} \varepsilon^2\log\varepsilon+o(\varepsilon^2\log\varepsilon),\ \varepsilon\to 0.
\end{split}
\end{equation}
Here, $h_i=h_i(k,A,a_2,\ldots,a_i)$, $i=2,\ldots,k$, are real-valued functions of $k$ and first $i$ coefficients of $f(z)$.
\end{lemma}

\begin{proof}
By Proposition~\ref{nhood}.$ii)$ in the Appendix, the area of the nucleus can be computed by adding areas of infinitely many crescent-shaped contributions. Furthermore, Proposition~\ref{crescent} provides the formula for computing such areas. We have
\begin{equation}\label{nepsum}
A(N_\varepsilon)
=\varepsilon^2\pi+2\varepsilon^2 \sum_{n=n_{\varepsilon}}^{\infty}\left(\frac{d_n}{2\varepsilon}\sqrt{1-\frac{d_n^2}{4\varepsilon^2}}+\arcsin{\frac{d_n}{2\varepsilon}}\right).
\end{equation}
By Proposition~\ref{auxsum} in the Appendix, this sum can be replaced by the following integral:
\begin{equation}\label{int}
A(N_\varepsilon)= 2\varepsilon^2\int_{x=n_\varepsilon}^{\infty}\left(\frac{d(x)}{2\varepsilon}\sqrt{1-\frac{d(x)^2}{4\varepsilon^2}}+\arcsin{\frac{d(x)}{2\varepsilon}}\right) dx +\text{O}(\varepsilon^2),\ \varepsilon\to 0,
\end{equation}
where $d(x)$ is the strictly decreasing function from Proposition~\ref{auxsum}:\\ $$d(x)=q_1x^{-1-\frac{1}{k}}+q_2x^{-1-\frac{2}{k}}+\ldots+q_kx^{-2}+q_{k+1}x^{-2-\frac{1}{k}}\log x+Dx^{-2-\frac{1}{k}}.$$
\smallskip

We now compute the first $k+1$ terms in the asymptotic development of the integral from \eqref{int}, as $\varepsilon\to 0$. Applying the change of variables $t=\frac{d(x)}{2\varepsilon}$, we get
\begin{equation}\label{intch}
I=-2\varepsilon \int_{0}^{\frac{d(n_\varepsilon)}{2\varepsilon}}\left(t \sqrt{1-t^2}+\arcsin{t}\right) \frac{1}{d'(x(t))} dt.
\end{equation}

Here, $x(t)=d^{-1}(2\varepsilon t)$. Note that, for a given $\varepsilon$, $t$ is bounded in $[0,1)$. Therefore it holds that:
\begin{equation}\label{unif}(\varepsilon t)\to 0, \text{ as }\varepsilon\to 0,\text{ uniformly in }t.\end{equation}
The development of $x(t)=d^{-1}(2\varepsilon t)$, as $\varepsilon \to 0$, can be deduced using the already computed development for $n_\varepsilon=d^{-1}(2\varepsilon)$ in Lemma~\ref{asyneps}. We have that
\begin{equation}\label{ratio}
\begin{split}
&\frac{1}{d'(x(t))}=-\frac{k}{k+1}2^{-2+\frac{1}{k+1}}|a_1A^{k+1}|^{1-\frac{1}{k+1}}(\varepsilon t)^{-2+\frac{1}{k+1}}+p_2(\varepsilon t)^{-2+\frac{2}{k+1}}+\ldots+\\
&\ \ +p_k^{(1)} (\varepsilon t)^{-2+\frac{k}{k+1}}+p_k^{(2)}(\varepsilon t)^{-1}\log(\varepsilon t)+O\big((\varepsilon t)^{-1+\frac{1}{k+1}}\big),\quad t\in\left[0,\frac{d(n_\varepsilon)}{2\varepsilon}\right),\ \varepsilon \to 0,
\end{split}\end{equation}
where $p_i=p_i(k,A,a_2,\ldots,a_i)$, $i=2,\ldots,k$, are real-valued functions.

Using Proposition~\ref{auxi2} in the Appendix, we remove $\varepsilon$ from the boundary of $I$. The integral in \eqref{intch} is equal to
\begin{equation}\label{inte2}
I=-2\varepsilon\int_{0}^{1}(t\sqrt{1-t^2}+\arcsin{t})\frac{1}{d'(x(t))} dt +o(\log\varepsilon), \ \varepsilon\to 0.
\end{equation}

Now, substituting the development \eqref{ratio} in \eqref{inte2} and using \eqref{unif} to evaluate the last term, we get
\begin{equation}\label{enddev}
\begin{split}
I&\ =\left(\frac{2}{|a_1A^{k+1}|}\right)^{-2+\frac{1}{k+1}}\frac{k}{k+1}\cdot T_1\cdot \varepsilon^{-1+\frac{1}{k+1}}-2p_2\cdot T_2\cdot \varepsilon^{-1+\frac{2}{k+1}}-\\[0.2cm]
&\quad -\ldots-2p_k^{(1)}\cdot T_k\cdot\varepsilon^{-1+\frac{k}{k+1}}-2p_k^{(2)}\cdot S_{k+1}\cdot \log\varepsilon+o(\log\varepsilon),\ \varepsilon\to 0.
\end{split}
\end{equation}
Here, functions $p_i$ are real-valued functions from the development \eqref{ratio} and $S_{k+1}$ and $T_i$, $i=1,\ldots,k,$ are the following finite integrals:
\begin{align*}
T_i=&\int_{0}^{1}(t\sqrt{1-t^2}+\arcsin t)t^{-2+\frac{i}{k+1}} dt, \ i=1,\ldots,k,\\
S_{k+1}=&\int_{0}^{1}(t\sqrt{1-t^2}+\arcsin t)t^{-1}\log t dt.
\end{align*}
Since $T_1=\frac{(k+1)\sqrt \pi}{2k}\left(\frac{\Gamma(\frac{1}{2k+2})}{\Gamma(\frac{3}{2}+\frac{1}{2k+2})}-\sqrt\pi\right),$ combining \eqref{int}, \eqref{inte2} and \eqref{enddev}, we get the development \eqref{nucleus} for $A(N_\varepsilon)$.
\end{proof}

\begin{lemma}[Asymptotic development of the area of the tail]\label{asytail}
The area of the tail of the $\varepsilon$-neighborhood of the orbit has the following asymptotic development:
\begin{equation}\label{tail}
\begin{split}
&A(T_{\varepsilon})=\pi\Big(\frac{2}{|a_1A^{k+1}|}\Big)^{-1+\frac{1}{k+1}}\varepsilon^{1+\frac{1}{k+1}}+f_2\varepsilon^{1+\frac{2}{k+1}}+\ldots+f_k\varepsilon^{1+\frac{k}{k+1}}+\\
&\ \ +\Big[\pi\frac{k}{k+1}Re\Big(\Big(\frac{a_{k+1}}{a_1}-\frac{(k+1)a_1}{2}\Big)A^{k}\Big)+g(k,A,a_2,\ldots,a_k)\Big]\varepsilon^2\log\varepsilon+\\
&\hspace{9cm}+\text{o}(\varepsilon^2 \log\varepsilon),\ \varepsilon\to 0.
\end{split}
\end{equation}
Here, $f_i=f_i(k,A,a_2,\ldots,a_i)$, $i=2,\ldots,k$, are real-valued functions which depend only on $k$ and the first $i$ coefficients of $f(z)$. The function $g=g(k,A,a_2,\ldots,a_k)$ has the property that $g(k,A,0,\ldots,0)=0.$
\end{lemma}

\begin{proof}
Since the tail, by definition, consists of $n_\varepsilon-1$ disjoint $\varepsilon$-discs, we have that $|T_\varepsilon|=(n_\varepsilon-1)\cdot \varepsilon^2\pi$. The statement follows from \eqref{devneps}.
\end{proof}

\begin{lemma}[Asymptotic development of the center of mass of the nucleus]\label{massnucl}
Let $\mathbf{t}(N_\varepsilon)$ denote the center of mass of the nucleus of the $\varepsilon$-neighborhood. The following asymptotic development holds:
\begin{equation}\label{nuclcenter}
\begin{split}
\mathbf{t}(N_\varepsilon)\cdot A(N_\varepsilon)&=q_1\varepsilon^{1+\frac{2}{k+1}}
+q_2\varepsilon^{1+\frac{3}{k+1}}+q_3\varepsilon^{1+\frac{4}{k+1}}+\ldots+q_{k}\varepsilon^{2}+\\
&\qquad\ \  +q_{k+1}\varepsilon ^{2+\frac{1}{k+1}}\log\varepsilon+o(\varepsilon^{2+\frac{1}{k+1}}\log\varepsilon),\ \varepsilon\to 0.
\end{split}
\end{equation}
Here, $q_i=q_i(k,A,a_2,\ldots,a_i)$, $i=1,\ldots,k+1,$ are complex-valued functions which depend on $k$ and on the first $i$ coefficients of $f(z)$. More precisely,
\begin{align*}
q_1&=\frac{k\sqrt\pi}{2(k-1)}\left(\frac{\Gamma(\frac{1}{k+1})}{\Gamma(\frac{3}{2}+\frac{1}{k+1})}-\sqrt\pi\right)\left(\frac{2}{|a_1A^{k+1}|}\right)^{-1+\frac{2}{k+1}}\cdot A,\\
q_{k+1}&=-\frac{k\sqrt\pi}{2(k+1)^2}\left(\sqrt\pi-\frac{\Gamma(\frac{1}{2}+\frac{1}{2k+2})}{\Gamma(2+\frac{1}{2k+2})}\right)\left(\frac{2}{|a_1A^{k+1}|}\right)^{\frac{1}{k+1}}\cdot\nonumber\\
&\hspace{4cm}\cdot Im\big(\frac{a_{k+1}}{a_1^2}-\frac{k+1}{2}\big)\cdot A\cdot i+g(k,A,a_2,\ldots,a_k),
\end{align*}
where $g=g(k,A,a_2,\ldots,a_k)$ is a complex-valued function such that $g(k,A,0,\ldots,0)\break =0$.
\end{lemma}

\begin{proof}
By definition of the centre of mass and by Propositions~\ref{nhood}.$ii)$ and \ref{crescent} in the Appendix, we have that
\begin{equation*}
\begin{split}
\mathbf{t}&(N_\varepsilon)\cdot A(N_{\varepsilon})=z_{n_{\varepsilon}}\cdot\varepsilon^2\pi+\sum_{n=n_{\varepsilon}+1}^{\infty}A(D_n)\mathbf{t}(D_n)=\\
&=z_{n_{\varepsilon}}\cdot\varepsilon^2\pi+2\varepsilon^2\sum_{n=n_{\varepsilon}+1}^{\infty}\left(\frac{d_n}{2\varepsilon}\sqrt{1-\frac{d_n^2}{4\varepsilon^2}}+\arcsin{\frac{d_n}{2\varepsilon}}\right)z_{n}+\\
&\quad \qquad +\varepsilon^2\sum_{n=n_{\varepsilon}+1}^{\infty}\Big(\frac{d_n}{2\varepsilon}\sqrt{1-\frac{d_n^2}{4\varepsilon^2}}-\arcsin\sqrt{1-\frac{d_n^2}{4\varepsilon^2}}\Big)(z_n-z_{n+1}).
\end{split}
\end{equation*}
Here, $D_n$, $n\geq n_\varepsilon$, denote the contributions to the nucleus from the $\varepsilon$-discs of the points $z_n$.

We first show that
\begin{equation}\label{cmass1}
\mathbf{t}(N_\varepsilon)\cdot A(N_{\varepsilon})=2\varepsilon^2\sum_{n=n_{\varepsilon}+1}^{\infty}\left(\frac{d_n}{2\varepsilon}\sqrt{1-\frac{d_n^2}{4\varepsilon^2}}+\arcsin{\frac{d_n}{2\varepsilon}}\right)z_{n}+O(\varepsilon^{2+\frac{1}{k+1}}),
\end{equation}
as $\varepsilon\to 0$.
From \eqref{devneps} and \eqref{znasy}, $z_{n_{\varepsilon}}\cdot\varepsilon^2\pi=O(\varepsilon^{2+\frac{1}{k+1}})$, as $\varepsilon\to 0$. On the other hand,
by \eqref{znrasy}, we have that $z_n-z_{n+1}=O(n^{-\frac{k+1}{k}})$, as $n\to\infty$. Therefore, using boundedness of the term in parenthesis, integral approximation of the sum and then \eqref{devneps}, we get
\begin{equation*}
\begin{split}
&\Big|\varepsilon^2\sum_{n=n_{\varepsilon}+1}^{\infty}\Big(\frac{d_n}{2\varepsilon}\sqrt{1-\frac{d_n^2}{4\varepsilon^2}}-\arcsin\sqrt{1-\frac{d_n^2}{4\varepsilon^2}}\Big)(z_n-z_{n+1})\Big|\leq \\
&\hspace{4cm}\leq C_1\varepsilon^2\sum_{n=n_\varepsilon+1}^{\infty}n^{-\frac{k+1}{k}}\leq C_2\varepsilon^2 n_\varepsilon^{-\frac{1}{k}}\leq C \varepsilon^{2+\frac{1}{k+1}},
\end{split}
\end{equation*} for some constant $C>0$. This proves \eqref{cmass1}.

To compute the first $k+1$ terms in the asymptotic development of the sum in \eqref{cmass1},
\begin{equation}\label{ssum2}
S=\sum_{n=n_{\varepsilon}+1}^{\infty}\left(\frac{d_n}{2\varepsilon}\sqrt{1-\frac{d_n^2}{4\varepsilon^2}}+\arcsin{\frac{d_n}{2\varepsilon}}\right)z_{n},
\end{equation} as $\varepsilon\to 0$, we use the same idea as in Lemma~\ref{asynucl}. Therefore we omit the details. To make the integral approximation of the sum $S$, we have to cut off the formal developments $d_n$ and $z_n$ to finitely many terms. Let $d_n^*$ be as in Proposition~\ref{auxsum}, $d_n^*=J_{k+1} d_n+Dn^{-2-\frac{1}{k}}$. By $J_{k+1} z_n$, we denote the first $k+1$ terms in the asymptotic development of $z_n$. It can be shown similarly as before that
$$
S=\sum_{n=n_\varepsilon+1}^\infty \Big[\Big(\frac{d_n^*}{2\varepsilon}\sqrt{1-\frac{(d_n^*)^2}{4\varepsilon^2}}+\arcsin{\frac{ d_n^*}{2\varepsilon}}\Big) J_k z_n\Big]+o(\varepsilon^{\frac{1}{k+1}}\log\varepsilon).
$$
Since the real and the imaginary part of the function under the summation sign are strictly decreasing, as $n\to\infty$, we can make the integral approximation of the sum:
\begin{equation}\label{int2}
S=\int_{n_\varepsilon}^{\infty}\Big(\frac{d(x)}{2\varepsilon}\sqrt{1-\frac{d(x)^2}{4\varepsilon^2}}+\arcsin{\frac{d(x)}{2\varepsilon}}\Big)z(x) dx+o(\varepsilon^{\frac{1}{k+1}}\log\varepsilon).
\end{equation}
The function $d(x)$ is as defined in \eqref{dx} of the Appendix, and $z(x)$ is equal to
\begin{equation}\label{xz}
z(x)=g_1x^{-\frac{1}{k}}+g_2 x^{-\frac{2}{k}}+g_3x^{-\frac{3}{k}}+g_4 x^{-\frac{4}{k}}+\ldots+g_k x^{-1}
+g_{k+1} x^{-\frac{k+1}{k}}\log x,
\end{equation}
with coefficients $g_i\in\mathbb{C}$ from the development \eqref{znasy} of $z_n$ in the Appendix.

By making the change of variables $t=\frac{d(x)}{2\varepsilon}$ in the integral, we get
\begin{equation}\label{int3}
S=-2\varepsilon \int_{0}^{1+O(\varepsilon^{1-\frac{1}{k+1}})}\Big(t\sqrt{1-t^2}+\arcsin t\Big)\frac{z(x(t))}{d'(x(t))} dt+o(\varepsilon^{\frac{1}{k+1}}\log\varepsilon),
\end{equation}
as $\varepsilon\to 0$.

Using  \eqref{ratio}, \eqref{xz} and the development for $x(t)$ from the proof of Lemma~\ref{asynucl}, after some computation we get the development for $\frac{z(x(t))}{d'(x(t)}$, as $\varepsilon t\to 0$.
Again, let us note that $\varepsilon t \to 0$ uniformly in $t$, as $\varepsilon\to 0$, see \eqref{unif} before. Substituting the development in \eqref{int3} and proceeding in a similar way as in Lemma~\ref{asynucl}, we get the development \eqref{nuclcenter}.
\end{proof}

\begin{lemma}[Development of the center of mass of the tail]\label{masstail}
The following development for the center of the mass of the tail of the $\varepsilon$-neighborhood holds:
\begin{equation}\label{tailcenter}
\begin{split}
&\mathbf{t}(T_\varepsilon)A(T_\varepsilon)=\frac{k}{k-1}\pi\left(\frac{2}{|a_1A^{k+1}|}\right)^{-1+\frac{2}{k+1}}\cdot A\cdot \varepsilon^{1+\frac{2}{k+1}}
+g_2\varepsilon^{1+\frac{3}{k+1}}+\\
&\hspace{4cm}+\ldots
+g_{k-1}\varepsilon^{1+\frac{k}{k+1}}+g_k \varepsilon^2\log\varepsilon+S(f,z_0)\varepsilon^2-\\
&-\left[\frac{\pi}{k+1}\left(\frac{2}{|a_1A^{k+1}|}\right)^{\frac{1}{k+1}}Im\big(\frac{a_{k+1}}{a_1^2}-\frac{k+1}{2}\big)\cdot i\cdot A+h(k,A,a_2,\ldots,a_{k})\right]\cdot\\
&\hspace{6cm}\cdot\varepsilon^{2+\frac{1}{k+1}}\log\varepsilon+o(\varepsilon^{2+\frac{1}{k+1}}\log\varepsilon),\ \varepsilon\to 0.
\end{split}
\end{equation}
Here, $g_i=g_i(k,A,a_2,\ldots,a_i)$, $i=2,\ldots, k$, are complex-valued functions of $k,\ A$ and first $i$ coefficients of $f(z)$. The function $S(f,z_0)$ is a complex-valued function which depends on the whole $f(z)$ and on the initial condition $z_0$. The function $h=h(k,A,a_2,\ldots,a_k)$ is a complex-valued function which satisfies $h(k,A,0,\ldots,0)=0$.\end{lemma}
\begin{proof}
\begin{equation*}
\begin{split}
\mathbf{t}(T_\varepsilon)\cdot A(T_{\varepsilon}&)=\frac{\sum_{n=1}^{n_\varepsilon-1}z_n\cdot \varepsilon^2\pi}{A(T_\varepsilon)}\cdot A(T_\varepsilon)=\varepsilon^2\pi\sum_{n=1}^{n_\varepsilon-1}z_n=\\
&=\varepsilon^2\pi(z_0+\ldots z_{n(f,z_0)})+\varepsilon^2\pi\cdot\sum_{n=n(f,z_0)}^{n_\varepsilon-1}z_n.
\end{split}
\end{equation*}
Here, $n(f,z_0)$ is chosen to be the first index, obviously depending on the diffeomorphism $f$ and on the initial condition $z_0$, such that  $$z_n=J_{k+1}z_n+R(n),\text{ where }|R(n)|\leq Cn^{-1-\frac{1}{k}},\text{ for }n\geq n(f,z_0),$$ for some constant $C>0$. Then
\begin{equation}\label{sumsum}
\begin{split}
\sum_{n=n(f,z_0)}^{n_\varepsilon-1}z_n=g_1\sum_{n=n(f,z_0)}^{n_\varepsilon-1}n^{-\frac{1}{k}}  +&g_2\sum_{n=n(f,z_0)}^{n_\varepsilon-1}n^{-\frac{2}{k}}+\ldots   +g_k\sum_{n=n(f,z_0)}^{n_\varepsilon-1}n^{-1}+\\
&\quad +g_{k+1}\sum_{n=n(f,z_0)}^{n_\varepsilon-1}n^{-1-\frac{1}{k}}\log n+ \sum_{n=n(f,z_0)}^{n_\varepsilon-1}R(n),
\end{split}
\end{equation}
where complex numbers $g_i$ are as in the development of $z_n$, see Proposition~\ref{orbit} in the Appendix.

We now compute the first $k+1$ terms in the asymptotic developments of \eqref{sumsum}, as $n_\varepsilon\to\infty$.

Firstly, we concentrate on the last sum in \eqref{sumsum}. We show that
\begin{equation}\label{residu}
\sum_{l=n(z_0,f)}^n R(l)=C(z_0,f)+O(n^{-\frac{1}{k}}), \ n\to\infty,
\end{equation}
where $R(l)=O(l^{-1-\frac{1}{k}})$, as $l\to\infty$, and $C(z_0,f)$ is a complex constant depending on the diffeomorphism and on the initial condition. From the asymptotics of $R(l)$, the sum $\sum_{l=n(z_0,f)}^\infty R(l)$ is obviously convergent and equal to some constant $C(z_0,f)\in\mathbb{R}$. We write
\begin{equation*}
\sum_{l=n(z_0,f)}^n R(l)=\sum_{l=n(z_0,f)}^\infty R(l)-\sum_{l=n}^\infty R(l)=C(z_0,f)+O(n^{-\frac{1}{k}}),\ n\to\infty,
\end{equation*} where the second sum is evaluated as $O(n^{-1/k})$ by integral approximation of the sum.

Secondly, we estimate first three terms in the asymptotic developments of the first $k+1$ sums in \eqref{sumsum}, as $n\varepsilon\to\infty$.
We show the procedure on the first sum. Let $$F(n)=\sum_{l=n(f,z_0)}^{n}l^{-\frac{1}{k}}.$$
Obviously, it satisfies the recurrence relation $$F(n+1)-F(n)=(n+1)^{-\frac{1}{k}}, \ n\in\mathbb{N},$$ with initial condition $F(n(f,z_0))=n(f,z_0)^{-\frac{1}{k}}$. We determine the first term in its development by integral approximation: $$F(n)=\frac{k}{k-1}n^{\frac{k-1}{k}}+R(n),$$ where $R(n)=o(n^\frac{k-1}{k})$ as $n\to\infty$. Using this development, from recurrence relation for $F(n)$ we get the recurrence relation for $R(n)$ and the initial condition $R(n(z_0,f))$. By recursion, we get $$R(n)=R(n(z_0,f))+\sum_{l=n(z_0,f)}^n O(l^{-1-\frac{1}{k}}).$$
Using \eqref{residu}, we conclude that $\sum_{l=n(f,z_0)}^{n}l^{-\frac{1}{k}}=C(n(z_0,f))+O(n^{-1/k})$, $n\to\infty$. The same procedure can be repeated for other sums.

Thus we obtain the development of the sum $\sum_{n=n(z_0,f)}^{n_\varepsilon-1}z_n$, as $n_\varepsilon\to\infty$. Substituting $n_\varepsilon$ with the development \eqref{devneps}, we get the development \eqref{tailcenter}, as $\varepsilon\to 0$.
\end{proof}

\noindent\emph{Proof of Theorem~\ref{asy}.} The proof follows from Lemmas \ref{asyneps} to \ref{masstail}, as described at the beginning of the Section.\qed

\section{Proof of Theorems~\ref{fnf} and \ref{fnfe}}\label{sec33}
We now prove the main results, Theorem~\ref{fnf} and Theorem~\ref{fnfe}.
In the proofs of Theorems~\ref{fnf} and \ref{fnfe}, we need the following lemma. It shows that the leading exponent and the relevant first and $(k+1)$-st coefficient in the development of the directed area remain unchanged by a change of variables tangent to the identity, transforming the diffeomorphism to its extended formal normal form.

\begin{lemma}[Invariance of fractal properties in the extended formal class]\label{inv}\ Let $f_1(z)$ and $f_2(z)$ be two germs of parabolic diffeomorphisms which belong to the same extended formal class $(k,a,a_1)$. Then it holds:
\begin{align*}
\dim_{\text{\it B}}(S^{f_1}(w_0))&=\dim_\text{\it B}(S^{f_2}(z_0)),\\
\mathcal{M}^\mathbb{C}(S^{f_1}(w_0))&=\mathcal{M}^\mathbb{C}(S^{f_2}(z_0)),\\ \mathcal{R^\mathbb{C}}(S^{f_1}(w_0))&=\mathcal{R^\mathbb{C}}(S^{f_2}(z_0)).
\end{align*}
Here, $z_0$ and $w_0$ are any two initial points chosen from the attracting sectors of $f_1$ and $f_2$ with the same attracting direction.
\end{lemma}

In the proof of Lemma~\ref{inv} we need the following auxiliary lemma:
\begin{lemma}\label{auxinv}
Let $f(z)$ be a parabolic diffeomorphism and let $g(z)=\phi_l^{-1}\circ f\circ \phi_l(z)$, where $\phi_l(z)=z+cz^l$, $l\geq 2$. Let $S^f(z_0)=\{z_n\}$ be an attracting orbit of $f(z)$ and let $S^g(v_0)=\{w_n=\phi_l(z_n)\}$ be the corresponding attracting orbit of $g(z)$. Then it holds that
\begin{equation}\label{orb}
K_1^{S^f(z_0)}=K_1^{S^g(v_0)},\ K_{k+1}^{S^f(z_0)}=K_{k+1}^{S^g(v_0)},
\end{equation}
where $K_1$ and $K_{k+1}$ denote the first and the $(k+1)$-st coefficients in the asymptotic developments \eqref{comarea} of the directed areas of the $\varepsilon$-neighborhoods of the corresponding orbits.
Furthermore, the equalities \eqref{orb} hold also if $S^f(z_0)$ and $S^g(v_0)$ are any two orbits of $f(z)$ and $g(z)$ respectively which converge to the same attracting direction.
\end{lemma}

\begin{proof}
Let $\{z_n\}$ be an attracting orbit of $f(z)$. We first take $S^g(v_0)=\{w_n\}$ to be the image of $\{z_n\}$ under $\phi_l,\ l>1$.
Using development \eqref{znasy} for $z_n$, we compute the development of $w_n=\phi_l(z_n)$. It is easy to see that, since $l>1$, the first coefficient and the $(k+1)$-st coefficient remain the same as in $z_n$, while the other coefficients can change. In particular, the attracting direction $A$ for $S^g(v_0)$ remains the same as for $S^f(z_0)$.

On the other hand, it can be seen in Section~\ref{sec3} that only the first and the $(k+1)$-st coefficient of the development of $z_n$ participate in the first and the $(k+1)$-st coefficient of the developments of $z_n-z_{n+1}$, $d_n$ and $n_\varepsilon$. Finally, the first and the $(k+1)$-st coefficient in the development of $A^{\mathbb{C}}(S^f(z_0)_\varepsilon)$, $K_1$ and $K_{k+1}$, depend only on the first and the $(k+1)$-st coefficient in the development of $z_n$, not on other coefficients. Therefore, the two coefficients remain unchanged in the change of variables $\phi_l(z)$.	 Finally, since $K_1$ and $K_{k+1}$ do not depend on the choice of initial point $z_0$ and $v_0$ inside one sector, we can choose any two orbits of the initial and of the transformed diffeomorphism converging to the same attracting direction $A$.
\end{proof}

\noindent\emph{Proof of Lemma~\ref{inv}.}
For the given diffeomorphisms $f_1(z)$ and $f_2(z)$, let \linebreak $\phi^{1,2}=\phi_{k}^{1,2}\circ \phi_{k-1}^{1,2}\circ\ldots\circ \phi_2^{1,2}$ denote the changes of variables obtained by composition of $k-1$ transformations of the above type, which present the first $k-1$ steps in transforming $f_1$ and $f_2$ to their extended formal normal forms. Let
\begin{align}\label{jets}
g_1=&(\phi^1)^{-1}\circ f_1 \circ \phi^1=z+a_1 z^{k+1}+a_1^2 a z^{2k+1}+\ldots,\\
g_2=&(\phi^2)^{-1}\circ f_2 \circ \phi^2=z+a_1 z^{k+1}+a_1^2 a z^{2k+1}+\ldots.\nonumber
\end{align}
Obviously, by Lemma~\ref{auxinv}, it holds that:
\begin{equation}\label{e1}
K_1^{g_1}=K_1^{f_1},\ K_1^{g_2}=K_1^{f_2} \text{ and } K_{k+1}^{g_1}=K_{k+1}^{f_1},\ K_{k+1}^{g_2}=K_{k+1}^{f_2},
\end{equation}
for the orbits corresponding to the same attracting direction. The notation $K_1^{f}$ is a bit imprecise, since the value differs for orbits in $k$ sectors, but we use it for simplicity and keep in mind that we always consider orbits converging to the same attracting direction.

\noindent Let $g_0$ be the extended formal normal form, $g_0(z)=z+a_1 z^{k+1}+a_1^2 a z^{2k+1}$. By further changes of variables, transforming $g_1$ and $g_2$ to the extended formal normal form $g_0$, the $(2k+1)$-jets from \eqref{jets} remain the same. Therefore we have, by the development \eqref{impcoef} in Theorem~\ref{asy}, that
\begin{equation}\label{e2}
K_1^{g_1}=K_1^{g_0},\ K_1^{g_2}=K_1^{g_0} \text{ and }  K_{k+1}^{g_1}=K_{k+1}^{g_0},\ K_{k+1}^{g_2}=K_{k+1}^{g_0},
\end{equation}
for the orbits corresponding to the same attracting direction. By \eqref{e1} and \eqref{e2}, it follows that $K_1^{f_1}=K_1^{f_2}$ and $K_{k+1}^{f_1}=K_{k+1}^{f_2}$, for the orbits of $f_1$ and $f_2$ converging to the same attracting direction.

Finally,  changes of variables do not change the multiplicity $k+1$ of the diffeomorphism. Therefore the leading exponent of the directed areas for all the orbits equals $1-\frac{1}{k+1}$.

Relating the coefficients $K_1$, $K_{k+1}$ and exponent $k$ with fractal properties of orbits, by \eqref{fra1} and \eqref{fra2}, the statement follows.
\qed

\smallskip
Note that the statement of the above Lemma is no longer true if we admit changes of variables which are not tangent to the identity. Only box dimension is then preserved.
\bigskip

\noindent \emph{Proof of Theorem~\ref{fnfe}.}
Let $f(z)=z+a_1z^{k+1}+o(z^{k+1})$ be a parabolic germ and let $g_0(z)=z+a_1z^{k+1}+a_1^2\cdot a\cdot z^{2k+1}$ be its extended formal normal form. Let $S^f(z_0)$ be an attracting orbit of $f(z)$ and let $S^{g_0}(w_0)$ be an attracting orbit of $g_0(z)$ with the same attracting direction.

The bijective correspondence between $k$ and $\dim_B(S^f(z_0))$ is obvious by \eqref{fra1}. Let $k$ then be fixed. Applying formulas \eqref{impcoef} from Theorem~\ref{asy} to the orbit of the formal normal form $g_0(z)$, we get the following formulas:
\begin{align}\label{impcoeffnf}
&K_1^{g_0}=\frac{k+1}{k}\cdot\sqrt\pi\cdot\frac{\Gamma(1 + \frac{1}{2k+2})}{\Gamma(\frac{3}{2}+\frac{1}{2k+2})}\bigg(\frac{2}{|a_1|}\bigg)^{1/(k+1)}\cdot \nu_A,\\
&K_{k+1}^{g_0}=\nu_A\cdot\Bigg[-\frac{\pi}{k+1}Re\big(a-\frac{k+1}{2}\big)+\Bigg.\nonumber\\
&\qquad \Bigg. \bigg(\frac{2(k-1)}{k+1}\Big(\frac{|a_1|}{2}\Big)^{1/(k+1)}\frac{\frac{\Gamma(\frac{1}{2}+\frac{1}{2k+2})}{\Gamma(2+\frac{1}{2k+2})}-\sqrt{\pi}}{\frac{\Gamma(\frac{1}{k+1})}{\Gamma(\frac{3}{2}+\frac{1}{k+1})}+\sqrt{\pi}}\bigg)\cdot i\cdot Im\big(a-\frac{k+1}{2}\big) \Bigg].\nonumber
\end{align}
By Lemma~\ref{inv},
\begin{equation}\label{co1}
K_1^f=K_1^{g_0},\ K_{k+1}^f=K_{k+1}^{g_0}.
\end{equation}
On the other hand, by \eqref{fra1} and \eqref{fra2},
\begin{equation}\label{co2}
K_1^{f}=\mathcal{M}^{\mathbb{C}}(S^f(z_0)),\ K_{k+1}^{f}=\mathcal{R}^{\mathbb{C}}(S^f(z_0)).
\end{equation}
Using \eqref{co1} and \eqref{co2}, we see that formulas \eqref{form} and \eqref{a1} in Theorem~\ref{fnfe} are just reformulations of \eqref{impcoeffnf}. They give, for a fixed $k$, the bijective correspondence between the pairs $(a_1,a)$ and $(\mathcal{M}^{\mathbb{C}}(S^f(z_0)),\ \mathcal{R}^{\mathbb{C}}(S^f(z_0)))$.
\qed
\bigskip

\noindent \emph{Proof of Theorem~\ref{fnf}.} Let $f(z)$ and $g_0(z)$ be as in the above proof. The standard formal normal form $f_0(z)$ is given by $f_0(z)=z+z^k+az^{2k+1}$, where $a$ is the same as in the extended form $g_0(z)$. The normal form $f_0(z)$ is obtained from $g_0(z)$ by making one extra change of variables of the type
$$
\phi(z)=a_1^{-1/k}z,
$$
in order to make coefficient $a_1$ equal to $1$. Since $a$ and $k$ are the same as in $g_0(z)$, formulas \eqref{form} expressing $k$ and $a$ from the fractal properties of the orbit $S^f(z_0)$ have already been obtained in the proof of Theorem~\ref{fnf}. Therefore the standard formal normal form of a diffeomorphism, described by the pair $(k,a)$, can be deduced from fractal properties $(\dim_B(S^f(z_0)),\ \mathcal{M}^{\mathbb{C}}(S^f(z_0),\ \mathcal{R}^{\mathbb{C}}(S^f(z_0)))$ of just one orbit of the diffeomorphism.
\qed
\medskip

Let us note that Theorem~\ref{fnf} cannot be formulated as an equivalence statement between $(k,a)$ and fractal properties. From the pair $(k,a)$, one cannot uniquely determine the fractal properties of the orbit of the initial diffeomorphism $f(z)$. Aside from the box dimension, the \emph{diffeomorphisms from the same standard formal class do not share the same fractal properties}. By \eqref{form} in Theorem~\ref{asy}, the directed Minkowski and residual content depend on the first coefficient $a_1$. The information on the initial fractal properties is lost by making changes of variables which are not tangent to the identity and which change $a_1$.

\section{Perspectives}\label{sec4}
Let us comment shortly on the perspectives for further research.

\subsection{Problem of analytic classification}\

We have shown in this article that the formal type of a diffeomorphism can be read from any orbit, using only its fractal properties. We are interested in prospects of \emph{analytic classification of parabolic diffeomorphisms using $\varepsilon$-neighborhoods of orbits}. The analytic classification was given independently by Ecalle and Voronin in \cite{ecalle} and \cite{voronin}. The analytic classes are given by the formal invariants $(k,a)$, as well as by $2k$ diffeomorphisms, which are called \emph{Ecalle-Voronin functional moduli of analytic classification}.

In general, a diffeomorphism is analitically conjugate to its formal normal form only sectorially. One orbit of a diffeomorphism lies completely in one sector. Our goal is to see if the analytic type can be read from the directed area of the $\varepsilon$-neighborhoods of only one orbit, or if perhaps something can be said in this directon if we consider $\varepsilon$-neigborhoods of one orbit per sector for $2k$ sectors, and compare them in an appropriate way.

\subsection{Can one hear the shape of a drum?}
In this article we were motivated by the question: to what extent a parabolic diffeomorphism itself can be reconstructed from its one realization, that is, from its one orbit? So far, we know that, from only one orbit, we can tell the formal type of a diffeomorphism. The concepts are somewhat similar to the concepts of the famous problem:\emph{ Can one hear a shape of a drum?}, presented by M. Kac in 1966. The question that is posed is if one can reconstruct the equation from only one solution, or, if not completely, how much can be said.

The vibrations of a drum are given by the Laplace equation with zero boundary condition on a given domain $\Omega$. The domain of the equation is the only unknown in the problem. The eigenvalues of the Laplace operator, $0<\lambda_1\leq \lambda_2\leq\ldots$, $\lambda_i\to\infty$,  present the frequencies. They are coefficients in the Fourier development of the solution. One tries to reconstruct the domain of the equation from these eigenvalues.

Let $N(\lambda)=\{\lambda_i:\lambda_i<\lambda\}$ be the eigenvalue counting function for the Laplace operator on $\Omega$. It was conjectured that from the asymptotic development of $N(\lambda)$, as $\lambda\to\infty$, one can obtain some properties of the domain:
\begin{conjecture}[Modified Weyl-Berry conjecture, Conjecture 5.1 in \cite{lapi}]
If $\Omega\subset \mathbb{R}^N$ has a Minkowski measurable boundary $\Gamma$, with box dimension $d\in(N-1,N)$, then
$$
N(\lambda)=(2\pi)^{-N}\mathcal{B}_N\cdot A(\Omega)\lambda^{N/2}+c_{N,d}\ \cdot \mathcal{M}(\Gamma)\cdot \lambda^{\frac{d}{2}}+o(\lambda^{\frac{d}{2}}),\ \lambda\to\infty.
$$
\end{conjecture}
Here, $\mathcal{B}_N$ is the volume of the unit ball in $\mathbb{R}^N$, $A(\Omega)$ the Lebesgue measure of the set $\Omega\in\mathbb{R}^N$ and $\mathcal{M}(\Gamma)$ the Minkowski content of the boundary. The constant $c_{N,d}$ is a real constant depending only on $N$ and $d$.
\smallskip

The conjecture was proven in the one-dimensional case, $N=1$, in Corollary 2.3 in \cite{lapi}. In other dimensions, it is still open.

Although we do not see the direct relation between two problems, in many aspects they appear similar. The general idea of reconstructing the equation from one solution is common, as well as the fact that, for obtaining more information on the equation, we need to utilize further terms in appropriate developments.

For more on this problem, see \cite{lapid} or \cite{lapi}.
\section{Appendix}\label{sec5}
Here we state auxiliary propositions that we need in the proof of Theorem~\ref{asy}.

%First let us recall a well-known result about integral approximation of the sum. The proof consists in bounding the integral by left Riemann sum from below and by right Riemann sum from above.
%\begin{proposition}[Integral approximation of the sum]\label{inte}
%Suppose $f(x)$ is monotonically decreasing, continuous function on the interval $[m-1,n]$, $m, n\in\{\mathbb{N}\cup\infty\},\ m<n$. Then the following inequality holds:\edz{DA LI TREBA REFERENCA ILI DOVOLJNO OBJASNJENO SA R. SUMAMA?}
%$$
%\int_{m}^{n}f(x)dx\leq \sum_{k=m}^{n}f(k)\leq \int_{m-1}^{n}f(x) dx.
%$$

%\end{proposition}
Let $f(z)=z+a_1z^{k+1}+a_2z^{k+2}+a_3z^{k+3}+\ldots$, $a_i\in\mathbb{C}$, $a_1\neq 0$, be a parabolic diffeomorphism. Let the initial point $z_0$ belong to an attracting sector. We denote by $A$ the attracting direction
$$
A=(-ka_1)^{-\frac{1}{k}},
$$
where we chose the one of $k$ complex roots for which $z_0$ is closest to the direction $A$.
\begin{proposition}[Asymptotic development of $z_n$]\label{orbit}
Let $z_n=f^{(\circ n)}(z_0)$, $n\in\mathbb{N}_0$, denote the points of the orbit $S^f(z_0)$. Then
\begin{equation}\label{znasy}
\begin{split}
z_n&=g_1n^{-\frac{1}{k}}+g_2 n^{-\frac{2}{k}}+g_3n^{-\frac{3}{k}}+g_4 n^{-\frac{4}{k}}+\ldots+g_k n^{-1}
+\\
&\hspace{5cm}+g_{k+1} n^{-\frac{k+1}{k}}\log n +o(n^{-\frac{k+1}{k}}\log n),
\end{split}
\end{equation}
where coefficients $g_i=g_i(k,A,a_2,\ldots,a_i)$, $i=1,\ldots,k+1$, are complex-valued functions of $k$ and first $i$ coefficients of $f(z)$.
More precisely,
$$
g_1=A,\ \ g_{k+1}=-\frac{1}{k}A^{k+1}\left(\frac{a_{k+1}}{a_1}-\frac{a_1(k+1)}{2}+h(k,A,a_2,\ldots,a_{k})\right),
$$
where $h=h(k,A,a_2,\ldots,a_{k})$ is a complex-valued function which satisfies \linebreak $h(k,A,0,\ldots,0)=0$.
\end{proposition}

\begin{proof} The following proof mimics the technique for obtaining the asymptotic development of a real iterative sequence from \cite{bruijn}, Chapter 8.4. In the complex case, we repeat the whole technique sectorially.  Suppose as above that $z_0$ lies in an attracting sector around attracting vector $A$. By \cite{milnor}, the whole trajectory $\{z_n\}$ lies in that attracting sector and is tangent to $A$ at the origin. On this sector, the change of variables
\begin{equation}\label{change}
z=A w^{-\frac{1}{k}}
\end{equation}
is well-defined, the complex root of $w$ being uniquely determined. The trajectory $\{z_n\}$ is transformed to $\{w_n\}$ and obviously
\begin{equation}\label{conv}
Arg(w_n^{-\frac{1}{k}})\to 0,\text{ as } n\to\infty.
\end{equation}
The recurrence relation for $z_n$
\begin{equation}\label{difference}
z_{n+1}=z_n+a_1z_n^{k+1}+a_2z_n^{k+2}+a_3z_n^{k+3}+\ldots
\end{equation}
transforms to the following recurrence relation for $w_n$:
\begin{equation}\label{chdiff}
\begin{split}
w_{n+1}=&w_n+1+\frac{a_2}{a_1}Aw_n^{-\frac{1}{k}}+\frac{a_3}{a_1}A^2w_n^{-\frac{2}{k}}+\ldots+\\
&+\frac{a_{k}}{a_1}A^{k-1}w_n^{-\frac{k-1}{k}}+\left(\frac{a_{k+1}}{a_1}-\frac{(k+1)a_1}{2}\right)A^{k}w_n^{-1}+o(w_n^{-1}).
\end{split}
\end{equation}
Obviously, $\frac{w_n-w_0}{n}=\frac{1}{n}\sum_{l=1}^{n}(w_l-w_{l-1})$. By \eqref{chdiff}, it holds that $(w_l-w_{l-1})\to 1$, as $l\to\infty$, therefore $\lim_{n\to\infty}\frac{w_n}{n}=1$. From \eqref{chdiff} we then have $$w_{n+1}-w_{n}=1+O(n^{-\frac{1}{k}}).$$ By recursion and using integral approximation of the sum, we get
$$
w_{n}=n+O(n^\frac{k-1}{k}).
$$
To compute the exact constant of the second term, with this development, we return to \eqref{chdiff} and get
$$
w_{n+1}-w_n=1+\frac{a_2}{a}A n^{-\frac{1}{k}}+O(n^{-\frac{2}{k}}).
$$
By recursion and using integral approximation of the sum,
$$
w_n=n+\frac{a_2}{a}A\frac{k}{k-1}n^{\frac{k-1}{k}}+O(n^\frac{k-2}{k}).
$$
Repeating this procedure $k$ times, we get the first $k+1$ terms in the development of $w_n$ as $n\to\infty$. The development for $z_n$ then follows from the development for $w_n$, \eqref{change} and \eqref{conv}.
\end{proof}

\medskip
The next two propositions give the tool for computing areas and centers of mass of $\varepsilon$-neighborhoods of orbits. Let $f(z)$ be a parabolic diffeomorphism and $S^f(z_0)=\{z_n,\ n\in\mathbb{N}_0\}$ its attracting orbit. Let $K(z_i,\varepsilon)$ denote the $\varepsilon$-disc centered at $z_i$.  We represent the $\varepsilon$-neighborhood of $S^f(z_0)$ as
\begin{equation}\label{union}
S^f(z_0)_\varepsilon=\bigcup_{i=0}^\infty D_i.
\end{equation}
Here, $D_0=K(z_0,\varepsilon)$ and $D_i=K(z_i,\varepsilon)\backslash \bigcup_{j=0}^{i-1} K(z_j,\varepsilon)$, $i\in\mathbb{N},$ are contributions from $\varepsilon$-discs of points $z_i$.

\begin{proposition}[Geometry of $\varepsilon$-neighborhoods of orbits]\ \label{nhood}

\begin{itemize}
\item[(i)] Distances between two consecutive points of the orbit, $|z_{n+1}-z_n|$, are, starting from some $n_0$, strictly decreasing as $n\to \infty$ .
\item[(ii)] For small enough $\varepsilon>0$,
$$
K(z_i,\varepsilon)\backslash \bigcup_{j=0}^{i-1}K(z_j,\varepsilon)=K(z_i,\varepsilon) \backslash K(z_{i-1},\varepsilon),\ i\in\mathbb{N}.
$$
\end{itemize}
\end{proposition}
Proposition~\ref{nhood}.$(ii)$ means that all contributions $D_i$ are in crescent or full-disc form, determined only by the distance to the previous point $z_{i-1}$. The positions of the points $z_0,\ldots,z_{i-2}$ do not affect the shape of $D_i$, see Figure~\ref{admiss}.

\begin{proof}
\emph{(i)}
Let us denote by $w_n=z_n-z_{n+1}-(z_{n+1}-z_{n+2})$. Using development \eqref{znrasy}, we compute: $$w_n=A\frac{k+1}{k^2}n^{-\frac{2k+1}{k}}+o(n^{-\frac{2k+1}{k}}),\ \ z_{n+1}-z_{n+2}=\frac{A}{k}n^{-\frac{k+1}{k}}+o(n^{-\frac{k+1}{k}}).$$  Obviously, in the limit as $n\to\infty$, the arguments of $w_n$ and $z_{n+1}-z_{n+2}$ are both equal to $Arg(A)$. For $n$ big enough, the value of the nonordered angle between $z_{n+1}-z_{n+2}$ and $w_n$ is therefore less than $\frac{\pi}{2}$. Since $z_n-z_{n+1}=(z_{n+1}-z_{n+2})+w_n$, it follows that $|z_n-z_{n+1}|>|z_{n+1}-z_{n+2}|$, for $n$ big enough.

\begin{figure}[htp]
\begin{center}
  % replace aims_logo.eps by your figure file name
  \includegraphics[width=3.5in]{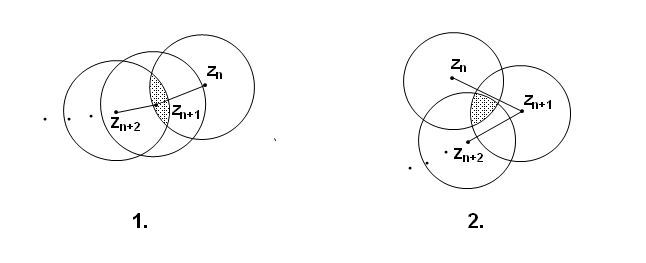}\\
  \caption{{\small 1. Admissible position of discs,\ 2. Nonadmissible position of discs.}}\label{admiss}
  \end{center}
\end{figure}

\emph{(ii)} Let $T_n$ denote the midpoint and $s_n$ the bisector of the segment $[z_{n+1},z_n]$, $n\in\mathbb{N}$. It will suffice to show that there exists $\varepsilon>0$ such that for every $n\geq n_0$ and for every $k\in\mathbb{N}$, the distance from the intersection of $s_n$ and $s_{n+k}$, denoted $S_{n,k}$, to the midpoint $T_n$ is greater than $\varepsilon$. In this way we ensure that the union of intersections of $\varepsilon$-disc of each new point of the orbit with the $\varepsilon$-discs of all the previous points is a subset of the intersection with the $\varepsilon$-disc of the previous point only.

We first show that the two consecutive bisectors $s_n$ and $s_{n+1}$ intersect at the distance from $T_n$ which is bounded from below by a positive constant, as $n\to\infty$. \\
\noindent The bisector $s_n$ can obviously be parametrized as follows
\begin{equation}\label{sim}
T_n+t\cdot i(z_n-z_{n+1})=\frac{z_n+z_{n+1}}{2}+t\cdot i(z_n-z_{n+1}),\ t\in\mathbb{R}.
\end{equation}
We denote by $t_n\in\mathbb{R}$ the parameter of the intersection $S_{n,1}$ of $s_n$ and $s_{n+1}$. The complex number $$\frac{z_n+z_{n+1}}{2}+t_n \cdot i(z_n-z_{n+1})-T_{n+1}=\frac{z_n-z_{n+2}}{2}+t_n\cdot i(z_n-z_{n+1})$$ is perpendicular to $z_{n+1}-z_{n+2}$. Therefore their scalar product, denoted by $(.|.)$, is equal to $0$, and we get:
{\small $$
t_n=-\frac{1}{2}\frac{(z_{n}-z_{n+1}|z_{n+1}-z_{n+2})+|z_{n+1}-z_{n+2}|^2}{Re(z_n-z_{n+1}) Im(z_{n+1}-z_{n+2})-Im(z_n-z_{n+1}) Re(z_{n+1}-z_{n+2})}.
%=&-\frac{1}{2}\cdot\frac{(z_{n}-z_{n+1}|z_{n+1}-z_{n+2})+|z_{n+1}-z_{n+2}|^2}{Re(z_n-z_{n+1})\cdot Im(z_{n+1}-z_{n+2})-Im(z_n-z_{n+1})\cdot Re(z_{n+1}-z_{n+2})}.\nonumber
$$}

\noindent Using development \eqref{znrasy},
after some computation, we get that the denominator is $O(n^{-\frac{3k+3}{k}})$, while
the numerator is $\frac{3|A|^2}{k^2}n^{-\frac{2k+2}{k}}+o(n^{-\frac{2k+2}{k}})$.  Therefore, $t_n\geq Cn^{\frac{k+1}{k}}$, for some positive constant $C>0$ and $n>n_0$. Since $|z_n-z_{n+1}|\simeq n^{-\frac{k+1}{k}}$,
the distance $$d(T_n,S_{n,1})=|t_n|\cdot|z_n-z_{n+1}|$$ is bounded from below by some positive constant for $n\geq n_0$, say by $M>0$.

It is left to show that the same lower bound holds not only for consecutive, but for any two bisectors $s_n$ and $s_{n+k}$, $k\in\mathbb{N}$, $n\geq n_0$. We can see from the development \eqref{znrasy} that the points of the orbit approach the origin in the direction $A$.
We draw the stripe of width $M/2$ on both sides of that tangent direction. Obviously, for $n$ big enough, no two bisectors can intersect inside the stripe without two consecutive bisectors being intersected inside the stripe, which is a contradiction with the first part. Therefore, the distances from the midpoints to the intersections of corresponding bisectors when $n\to\infty$ are uniformly bounded from below by e.g. $M/4$.

Taking $\varepsilon<M/4$, we have proven the statement.
\end{proof}

%The next proposition shows how to compute the areas and vectors of the center of mass of the crescents $D_i$ from Proposition~\ref{nhood}.
\begin{proposition}\label{crescent}
Let $z,\ w\in \mathbb{C}$ $($or $\mathbb{R}^2$$)$, $\varepsilon>0$. Suppose $|z-w|<2\varepsilon$. Let $D$ denote the crescent $D=K(z,\varepsilon)\backslash K(w,\varepsilon)$. Then its area is equal to
\begin{equation}\label{D}
A(D)=2\varepsilon^2 \left(\frac{|z-w|}{2\varepsilon}\sqrt{1-\frac{|z-w|^2}{4\varepsilon^2}}+\arcsin{\frac{|z-w|}{2\varepsilon}}\right),
\end{equation}
and its center of mass is equal to
\begin{equation}\label{MC}
\mathbf{t}(D)=z+\varepsilon^2(w-z)\frac{\frac{|z-w|}{2\varepsilon}\sqrt{1-\frac{|z-w|^2}{4\varepsilon^2}}-\arcsin\sqrt{1-\frac{|z-w|^2}{4\varepsilon^2}}}{A(D)}.
\end{equation}
\end{proposition}

\begin{proof} The proposition is proved by integration, $$A(D)=\int\int_{D} dx\ \!dy,\  \ \mathbf{t}(D)=\frac{1}{A(D)}\left(\int\int_{D} x\ \!dx\ \!dy\ + i\cdot\int\int_{D} y\ \!dx\ \!dy\right).$$
\end{proof}

The following two propositions are auxiliary results in the proof of \linebreak Lemma~\ref{asynucl}.
\begin{proposition}\label{auxsum}
The sum
\begin{equation}\label{ssum}
\sum_{n=n_{\varepsilon}}^{\infty}\left(\frac{d_n}{2\varepsilon}\sqrt{1-\frac{d_n^2}{4\varepsilon^2}}+\arcsin{\frac{d_n}{2\varepsilon}}\right)
\end{equation}
is equal to
$$ \int_{x=n_\varepsilon}^{\infty}\left(\frac{d(x)}{2\varepsilon}\sqrt{1-\frac{d(x)^2}{4\varepsilon^2}}+\arcsin{\frac{d(x)}{2\varepsilon}}\right) dx +O(1),\
$$
as $\varepsilon\to 0$, where $d(x)$ is given by
\begin{equation}\label{dx}
d(x)=q_1x^{-1-\frac{1}{k}}+q_2x^{-1-\frac{2}{k}}+\ldots+q_kx^{-2}+q_{k+1}x^{-2-\frac{1}{k}}\log x+Dx^{-2-\frac{1}{k}}.
\end{equation}
All the coefficients $q_i$ are the same as in development \eqref{dn} of $d_n$ and $D\in\mathbb{R}$ is some constant.
\end{proposition}

\begin{proof}
The idea is to apply integral approximation of the sum. The problem is that we only have formal asymptotic development of $d_n$. The idea is to cut off the formal asymptotic development at $(k+1)$-st term, to get a continuous and decreasing function of $n$ under the summation sign. We show here that the cut-off remainder is in some sense small and contributes to the sum with no more than $O(1)$, as $\varepsilon\to 0$.

We denote by $J_{k+1}d_n$ the first $k+1$ terms in the asymptotic expansion~\eqref{dn}.
For the sum with truncated $d_n$,
\begin{equation}
\sum_{n=n_{\varepsilon}}^{\infty}\left(\frac{J_{k+1} d_n}{2\varepsilon}\sqrt{1-\frac{(J_{k+1}d_n)^2}{4\varepsilon^2}}+\arcsin{\frac{J_{k+1}d_n}{2\varepsilon}}\right),
\end{equation}
to be well-defined, we have to ensure that $0<J_{k+1}d_n<2\varepsilon$ for $n\geq n_\varepsilon$. Since $d_n<2\varepsilon$ for $n\geq n_\varepsilon$ by \eqref{neps}, it is enough to achieve that $J_{k+1}d_n<d_n$, for $n\geq n_\varepsilon$, where $\varepsilon$ is sufficiently small. This is obtained by adding the term $Dn^{-2-\frac{1}{k}}$ to $J_{k+1}d_n$. Here, $D$ is chosen negative and sufficiently big by absolute value. We denote $d_n^*=J_{k+1}d_n+Dn^{-2-\frac{1}{k}}$. Obviously,
\begin{equation}\label{manip1}
d_n=d_n^*+O(n^{-2-\frac{1}{k}}).
\end{equation}
Let us denote the function under the summation sign in \eqref{ssum} by $h(x)$: $$h(x)=\frac{x}{2\varepsilon}\sqrt{1-\frac{x^2}{4\varepsilon^2}}+\arcsin{\frac{x}{2\varepsilon}}.$$ Then, $h'(x)=\frac{1}{\varepsilon}\sqrt{1-\big(\frac{x}{2\varepsilon}\big)^2}$. By \eqref{manip1} and by the mean value theorem, \begin{equation}\label{mvt}h(d_n)=h(d_n^*)+h'(\xi_n)\cdot O(n^{-2-\frac{1}{k}}),\ \xi_n\in[d_n^*,d_n].\end{equation}
Furthermore, \begin{equation}\label{h}
0<h'(\xi_n)<\frac{1}{\varepsilon}, \ n\geq n_\varepsilon.
\end{equation}
The initial sum \eqref{ssum} can, by \eqref{mvt}, be evaluated as follows:
\begin{equation}\label{ssum1}
S=\sum_{n=n_\varepsilon}^{\infty}h(d_n)=\sum_{n=n_\varepsilon}^\infty h(d_n^*)+\sum_{n=n_\varepsilon}^{\infty}h'(\xi_n)O(n^{-2-\frac{1}{k}}).
\end{equation}
By \eqref{h} and Lemma~\ref{asyneps}, using integral approximation of the sum, we get
\begin{equation}\label{st1}
|\sum_{n=n_\varepsilon}^{\infty}h'(\xi_n)
\cdot O(n^{-2-\frac{1}{k}})|<\frac{C_1}{\varepsilon}\sum_{n=n_\varepsilon}^{\infty}n^{-2-\frac{1}{k}}<\frac{C_2}{\varepsilon}n_\varepsilon^{-1-\frac{1}{k}}<C,
\end{equation}
for some constant $C>0$, as $\varepsilon\to 0$.

\noindent Furthermore, using integral approximation of the sum and the fact that the subintegral function is bounded from above, we get
\begin{equation}\label{st2}
\sum_{n=n_\varepsilon}^\infty h(d_n^*)=\int_{x=n_\varepsilon}^{\infty}\left(\frac{d(x)}{2\varepsilon}\sqrt{1-\frac{d(x)^2}{4\varepsilon^2}}+\arcsin{\frac{d(x)}{2\varepsilon}}\right) dx+O(1),\ \varepsilon\to 0.
\end{equation}
Finally, by \eqref{ssum1}, \eqref{st1} and \eqref{st2}, the result follows.
\end{proof}

\begin{proposition}\label{auxi2}
The integral
$$
\int_{0}^{\frac{d(n_\varepsilon)}{2\varepsilon}}\left(t \sqrt{1-t^2}+\arcsin{t}\right) \frac{1}{d'(x(t))} dt,
$$
where $d(x),\ x(t)$ and $n_\varepsilon$ are as in Lemma~\ref{asynucl}, is equal to the integral
$$
\int_{0}^{1}\left(t \sqrt{1-t^2}+\arcsin{t}\right) \frac{1}{d'(x(t))} dt+o(\varepsilon^{-1}),
$$
as $\varepsilon\to 0$.
\end{proposition}

\begin{proof}
We first show that the upper boundary $\frac{d(n_\varepsilon)}{2\varepsilon}$ in the integral is equal to
\begin{equation}\label{lower}
\frac{d(n_\varepsilon)}{2\varepsilon}=1+O(\varepsilon^{1-\frac{1}{k+1}}),\ \varepsilon\to 0.
\end{equation}
By \eqref{devneps} and \eqref{manip1},
\begin{equation}\label{n1}
d(n_\varepsilon)=d_{n_\varepsilon}^*=d_{n_\varepsilon}+O(\varepsilon^{2-\frac{1}{k+1}}).
\end{equation}
From \eqref{dn}, it can easily be seen that $d_{n+1}=d_n+O(n^{-2-\frac{1}{k}})$, thus by \eqref{devneps} and \eqref{neps}, we get
\begin{equation}\label{n2}
d_{n_\varepsilon}=2\varepsilon+O(\varepsilon^{2-\frac{1}{k+1}}).
\end{equation}
Combining \eqref{n1} and \eqref{n2}, \eqref{lower} follows.

Using \eqref{lower}, the above integral $I=\int_{0}^{\frac{d(n_\varepsilon)}{2\varepsilon}}\left(t \sqrt{1-t^2}+\arcsin{t}\right) \frac{1}{d'(x(t))} dt$ can be written as the sum
\begin{equation*}
I=\int_{0}^{1}(t\sqrt{1-t^2}+\arcsin{t})\frac{1}{d'(x(t))} dt+\int_{1}^{1+O(\varepsilon^{1-\frac{1}{k+1}})}(t\sqrt{1-t^2}+\arcsin{t})\frac{1}{d'(x(t))} dt.
\end{equation*}
By \eqref{ratio}, $\frac{1}{d'(x(t))}=O((\varepsilon t)^{-2+\frac{1}{k+1}})$. It is then easy to see that the second integral equals $O(\varepsilon^{-1})$, as $\varepsilon\to 0$, due to the boundedness of the subintegral function in the neighborhood of $t=1$.
\end{proof}

\section*{Acknowledgments}
I would like to thank my supervisors, Pavao Marde\v si\' c, for proposing the subject and for useful suggestions, and Vesna \v Zupanovi\' c, for many useful suggestions and discussions.

\medskip
\quad% The data information below will be filled by AIMS editorial staff
\medskip

\begin{thebibliography}{99}

\bibitem{bruijn} 
\newblock N. G. de Bruijn,
\newblock ``Asymptotic Methods in Analysis,"
\newblock North-Holland Publishing Co., Amsterdam, 1958.

\bibitem{ecalle} 
\newblock J. Ecalle,
\newblock ``Les Fonctions R\' esurgentes. Tome III,"
\newblock Publications Math\' ematiques d'Orsay, \textbf{85}, Universit\' e de Paris-Sud, D\' epartement de Mathematiques, Orsay, 1985.

\bibitem{falconer} 
\newblock K. Falconer,
\newblock ``Fractal Geometry: Mathematical Foundations and Applications,"
\newblock John Wiley and Sons Ltd., Chichester, 1990.

\bibitem{ilya} 
\newblock Y. Ilyashenko and S. Yakovenko,
\newblock ``Lectures on Analytic Differential Equations, Graduate Studies in Mathematics,"
\newblock \textbf{86} American Mathematical Society, Providence, RI, 2008, xiv+625.

\bibitem{klu} 
\newblock I. Kluvanek and G. Knowles,
\newblock ``Vector Measures and Control Systems,"
\newblock North-Holland Mathematics Studies 20, Amsterdam, 1976.

\bibitem{lapid} 
\newblock M. L. Lapidus,
\newblock \emph{Spectral and fractal geometry: From the Weyl-Berry conjecture for the vibrations of fractal drums to the Riemann zeta-function},
\newblock Differential Equations and Mathematical Physics (Birmingham, AL, 1990), Math. Sci. Engrg., \textbf{186} (1992), Academic Press, Boston, 151--181.

\bibitem{lapi} 
\newblock M. L. Lapidus and C. Pomerance,
\newblock \emph{The Riemann zeta-function and the one-dimensional Weyl-Berry conjecture for fractal drums},
\newblock Proceedings of the London Mathematical Society (3), \textbf{66} (1993), 41--69.

\bibitem{loray}
\newblock F. Loray,
\newblock ``Pseudo-Groupe D'une Singularit\'e de Feuilletage Holomorphe en Dimension Deux,"
\newblock Pr\' epublication IRMAR, ccsd-00016434, 2005.

\bibitem{MaReZu} 
\newblock P. Marde\v si\' c, M. Resman and V. \v Zupanovi\' c,
\newblock \emph{Multiplicity of fixed points and growth of $\varepsilon$-neighborhoods of orbits},
\newblock J. Differential Equations, \textbf{253} (2012), 2493--2514.

\bibitem{milnor} 
\newblock J. Milnor,
\newblock ``Dynamics in One Complex Variable, Introductory Lectures,"
\newblock 2$^{nd}$ edition, Friedr. Vieweg. \& Sohn Verlagsgesellschaft mbH, Braunschweig/Wiesbaden, 1999.

\bibitem{PT} 
\newblock J. Palis and F. Takens,
\newblock ``Hyperbolicity and Sensitive Chaotic Dynamics at Homoclinic Bifurcations,"
\newblock Cambridge University Press, 1993.

\bibitem{pratap}
\newblock R. Pratap and A. Ruina,
\newblock ``Introduction to Statistics and Dynamics,"
\newblock Pre-print for Oxford University Press, 2001.

\bibitem{tricot} 
\newblock C. Tricot,
\newblock ``Curves and Fractal Dimension,"
\newblock Springer-Verlag, Paris, 1993.

\bibitem{overview}
\newblock V. \v Zupanovi\' c and \ D. \v Zubrini\' c,
\newblock \emph{Fractal dimensions in dynamics},
\newblock in ``Encyclopedia of Mathematical Physics" \textbf{2} (2006), Elsevier, Oxford, 394--402.

\bibitem{voronin} 
\newblock S. M. Voronin,
\newblock \emph{Analytic classification of germs of conformal mappings $(\mathbb{C},0)\to(\mathbb{C},0)$},
\newblock Functional Anal. Appl., \textbf{15}(1981), 1--13.


\end{thebibliography}
\end{document}